\documentclass[12pt]{article}
 \usepackage[cp1251]{inputenc} 
 \usepackage[english,russian]{babel}

\usepackage{amsmath}
\usepackage{amssymb}
\usepackage{amsxtra}
\usepackage{latexsym}
\usepackage{ifthen}

\begin{document}

\newcounter{lemma}
\newcommand{\lemma}{\par \refstepcounter{lemma}%
{\bf ╦хььр \arabic{lemma}.}}

\newcounter{corollary}
\newcommand{\corollary}{\par \refstepcounter{corollary}%
{\bf ╤ыхфёЄтшх \arabic{corollary}.}}

\newcounter{remark}
\newcommand{\remark}{\par \refstepcounter{remark}%
{\bf ╟рьхўрэшх \arabic{remark}.}}

\newcounter{theorem}
\newcommand{\theorem}{\par \refstepcounter{theorem}%
{\bf ╥хюЁхьр \arabic{theorem}.}}

\newcounter{proposition}
\newcommand{\proposition}{\par \refstepcounter{proposition}%
{\bf ╧Ёхфыюцхэшх \arabic{proposition}.}}

\renewcommand{\refname}{\centerline{\bf ╤яшёюъ ышЄхЁрЄєЁ√}}

\newcommand{\proof}{{\it ─юърчрЄхы№ёЄтю.\,\,}}

\title{ {\leftline{\small ╙─╩ 517.5}}
╬с єёЄЁрэхэшш юёюсхээюёЄхщ ъырёёют ╬Ёышўр--╤юсюыхтр ё ъюэхўэ√ь
шёърцхэшхь}
\author{╧хЄЁют~┼.└., ╤рышьют~╨.╨., ╤хтюёЄ№ эют~┼.└.}
\maketitle
%
%



%
%
%
%
\maketitle
\begin{abstract}
╚чєўрхЄё  ыюъры№эюх яютхфхэшх чрьъэєЄю-юЄъЁ√Є√ї фшёъЁхЄэ√ї
юЄюсЁрцхэшщ ъырёёют ╬Ёышўр\,--\,╤юсюыхтр т ${\Bbb R}^n,$ $n\geqslant
3.$ ╙ёЄрэютыхэю, ўЄю єърчрээ√х юЄюсЁрцхэш  $f$ шьх■Є эхяЁхЁ√тэюх
яЁюфюыцхэшх т шчюышЁютрээє■ Єюўъє $x_0$ уЁрэшЎ√ юсырёЄш
$D\setminus\{x_0\},$ ъръ Єюы№ъю шї тэєЄЁхээ   фшырЄрЎш  яюЁ фър
$p\in (n-1, n]$ шьххЄ ьрцюЁрэЄє ъырёёр $FMO$ (ъюэхўэюую ёЁхфэхую
ъюыхсрэш ) т єърчрээющ Єюўъх ш, ъЁюьх Єюую, яЁхфхы№э√х ьэюцхёЄтр
юЄюсЁрцхэш  $f$ т $x_0$ ш эр $\partial D$ эх яхЁхёхър■Єё . ─Ёєушь
фюёЄрЄюўэ√ь єёыютшхь тючьюцэюёЄш эхяЁхЁ√тэюую яЁюфюыцхэш  єърчрээ√ї
юЄюсЁрцхэшщ  ты хЄё  ЁрёїюфшьюёЄ№ эхъюЄюЁюую шэЄхуЁрыр.
\end{abstract}

{\it ╩ы■ўхт√х ёыютр:} ьюфєыш ёхьхщёЄт ъЁшт√ї ш яютхЁїэюёЄхщ,
юЄюсЁрцхэш  ё юуЁрэшўхээ√ь ш ъюэхўэ√ь шёърцхэшхь, ъырёё√ ╤юсюыхтр ш
╬Ё\-ыш\-ўр--╤юсюыхтр, єёЄЁрэхэшх шчюышЁютрээ√ї юёюсхээюёЄхщ


\section{┬тхфхэшх} ┬ эрёЄю ∙хщ чрьхЄъх шёёыхфєхЄё  эхъюЄюЁ√щ
яюфъырёё юЄюсЁрцхэшщ ё ъюэхўэ√ь шёърцхэшхь, ръЄштэю шчєўрхь√ї т
яюёыхфэхх тЁхь  Ё фюь ртЄюЁют (ёь., эряЁ., \cite{IM},
\cite{MRSY}--\cite{MRSY$_1$}, \cite{GRSY}, \cite{GG}, \cite{GS} ш
\cite{Sal$_1$}). ┬ё■фє фрыхх $D$ -- юсырёЄ№ т ${\Bbb R}^n,$
$n\geqslant 2,$ $m$ -- ьхЁр ╦хсхур т ${\Bbb R}^n$ ш ${\rm
dist\,}(A,B)$ -- хтъышфютю ЁрёёЄю эшх ьхцфє ьэюцхёЄтрьш $A$ ш $B$ т
${\Bbb R}^n,$ $d(x, y):=|x-y|,$ $d(C)$ -- хтъышфют фшрьхЄЁ ьэюцхёЄтр
$C\subset{\Bbb R}^n,$
$$B(x_0, r)=\left\{x\in{\Bbb R}^n: |x-x_0|< r\right\}\,,\quad {\Bbb
B}^n := B(0, 1)\,,$$
$$S(x_0,r) = \{ x\,\in\,{\Bbb R}^n : |x-x_0|=r\}\,,\quad{\Bbb
S}^{n-1}:=S(0, 1)\,,$$
$$A(r_1,r_2,x_0)=\{ x\,\in\,{\Bbb R}^n : r_1<|x-x_0|<r_2\}\,,$$
$\omega_{n-1}$ юсючэрўрхЄ яыю∙рф№ хфшэшўэющ ёЇхЁ√ ${\Bbb S}^{n-1}$ т
${\Bbb R}^n,$ $\Omega_{n}$ -- юс·╕ь хфшэшўэюую °рЁр ${\Bbb B}^{n}$ т
${\Bbb R}^n,$ $\overline{{\Bbb R}^n}:={\Bbb R}^n\cup\{\infty\}.$ ┬
фры№эхщ°хь тё■фє ёшьтюыюь $\Gamma(E,F,D)$ ь√ юсючэрўрхь ёхьхщёЄтю
тёхї ъЁшт√ї $\gamma:[a,b]\rightarrow\overline{{\Bbb R}^n},$ ъюЄюЁ√х
ёюхфшэ ■Є ьэюцхёЄтр $E$ ш $F$ т $D,$ Є.х. $\gamma(a)\in
E,\,\gamma(b)\in F$ ш $\gamma(t)\in D$ яЁш $t\in(a,\,b).$ ╟ряшё№
$f:D\rightarrow {\Bbb R}^n$ яЁхфяюырурхЄ, ўЄю юЄюсЁрцхэшх $f$
эхяЁхЁ√тэю т $D.$
\medskip ┬ фры№эхщ°хь ${\mathcal H}^k$ -- эюЁьшЁютрээр  $k$-ьхЁэр 
ьхЁр ╒рєёфюЁЇр т ${\Bbb R}^n,$ $1\leqslant k\leqslant n,$ $J(x,
f)={\rm det}\, f^{\,\prime}(x)$ -- {\it  ъюсшрэ юЄюсЁрцхэш } $f$ т
Єюўъх $x,$ уфх $f^{\,\prime}(x)$ -- {\it ьрЄЁшЎр ▀ъюсш} юЄюсЁрцхэш 
$f$ т Єюўъх $x.$ ╟фхё№ ш фрыхх {\it яЁхфхы№э√ь ьэюцхёЄтюь
юЄюсЁрцхэш  $f$ юЄэюёшЄхы№эю ьэюцхёЄтр $E\subset \overline{{\Bbb
R}^n}$} эрч√трхЄё  ьэюцхёЄтю
$C(f, E):=\left\{y\in {\Bbb R}^n: \,\exists\,x_0\in E:
y=\lim\limits_{m\rightarrow \infty} f(x_m), x_m\rightarrow
x_0\right\}.$ ╬ЄюсЁрцхэшх $f:D\rightarrow {\Bbb R}^n$ эрч√трхЄё 
{\it ёюїЁрэ ■∙шь уЁрэшЎє юЄюсЁрцхэшхь} (ёь. \cite[Ёрчф. 3, уы.
II]{Vu$_1$}), хёыш т√яюыэхэю ёююЄэю°хэшх $C(f,
\partial D)\subset \partial f(D).$ ╬ЄюсЁрцхэшх $f:D\rightarrow {\Bbb R}^n$
эрч√трхЄё  {\it фшёъЁхЄэ√ь}, хёыш яЁююсЁрч $f^{-1}\left(y\right)$
ърцфющ Єюўъш $y\in{\Bbb R}^n$ ёюёЄюшЄ Єюы№ъю шч шчюышЁютрээ√ї Єюўхъ.
╬ЄюсЁрцхэшх $f:D\rightarrow {\Bbb R}^n$ эрч√трхЄё  {\it юЄъЁ√Є√ь},
хёыш юсЁрч ы■сюую юЄъЁ√Єюую ьэюцхёЄтр $U\subset D$  ты хЄё  юЄъЁ√Є√ь
ьэюцхёЄтюь т ${\Bbb R}^n.$ ╬ЄьхЄшь Єръцх, ўЄю т ёыєўрх хёыш
$f:D\rightarrow {\Bbb R}^n$ юЄъЁ√Єю ш фшёъЁхЄэю, Єю чрьъэєЄюёЄ№
юЄюсЁрцхэш  $f$ ¤ътштрыхэЄэр Єюьє ўЄю $f$ ёюїЁрэ хЄ уЁрэшЎє (ёь.
\cite[ЄхюЁхьр~3.3]{Vu$_1$}). ╧єёЄ№ $U$ -- юЄъЁ√Єюх ьэюцхёЄтю,
$U\subset {\Bbb R}^n,$ $u:U\rightarrow {\Bbb R}$ -- эхъюЄюЁр 
ЇєэъЎш , $u\in L_{loc}^{\,1}(U).$ ╧Ёхфяюыюцшь, ўЄю эрщф╕Єё  ЇєэъЎш 
$v\in L_{loc}^{\,1}(U),$ Єрър  ўЄю
$\int\limits_U \frac{\partial \varphi}{\partial x_i}(x)u(x)dm(x)=
-\int\limits_U \varphi(x)v(x)dm(x)$
фы  ы■сющ ЇєэъЎшш $\varphi\in C_1^{\,0}(U).$ ╥юуфр сєфхь уютюЁшЄ№,
ўЄю ЇєэъЎш  $v$  ты хЄё  {\it юсюс∙╕ээющ яЁюшчтюфэющ яхЁтюую яюЁ фър
ЇєэъЎшш $u$ яю яхЁхьхээющ $x_i$} ш юсючэрўрЄ№ ёшьтюыюь:
$\frac{\partial u}{\partial x_i}(x):=v.$ ╘єэъЎш  $u\in
W_{loc}^{1,1}(U),$ хёыш $u$ шьххЄ юсюс∙╕ээ√х яЁюшчтюфэ√х яхЁтюую
яюЁ фър яю ърцфющ шч яхЁхьхээ√ї т $U,$ ъюЄюЁ√х  ты ■Єё  ыюъры№эю
шэЄхуЁшЁєхь√ьш т $U.$

\medskip
╧єёЄ№ $G$ -- юЄъЁ√Єюх ьэюцхёЄтю т ${\Bbb R}^n.$ ╬ЄюсЁрцхэшх
$f:G\rightarrow {\Bbb R}^n$ яЁшэрфыхцшЄ {\it ъырёёє ╤юсюыхтр}
$W^{1,1}_{loc}(G),$ яш°єЄ $f\in W^{1,1}_{loc}(G),$ хёыш тёх
ъююЁфшэрЄэ√х ЇєэъЎшш $f=(f_1,\ldots,f_n)$ юсырфр■Є юсюс∙╕ээ√ьш
ўрёЄэ√ьш яЁюшчтюфэ√ьш яхЁтюую яюЁ фър, ъюЄюЁ√х ыюъры№эю шэЄхуЁшЁєхь√
т $G$ т яхЁтющ ёЄхяхэш. ╬ЄюсЁрцхэшх $f:D\rightarrow {\Bbb R}^n$
эрч√трхЄё  {\it юЄюсЁрцхэшхь ё ъюэхўэ√ь шёърцхэшхь}, хёыш $f\in
W_{loc}^{1,1}(D)$ ш фы  эхъюЄюЁющ ЇєэъЎшш $K(x): D\rightarrow
[1,\infty)$ т√яюыэхэю єёыютшх
$\Vert f^{\,\prime}\left(x\right) \Vert^{n}\leqslant K(x)\cdot
|J(x,f)|$
яЁш яюўЄш тёхї $x\in D,$ уфх $\Vert
f^{\,\prime}(x)\Vert=\max\limits_{h\in {\Bbb R}^n \setminus \{0\}}
\frac {|f^{\,\prime}(x)h|}{|h|}$ (ёь. \cite[я.~6.3, уы.~VI]{IM}.
╧юырурхь $l\left(f^{\,\prime}(x)\right)\,=\,\,\,\min\limits_{h\in
{\Bbb R}^n \setminus \{0\}} \frac {|f^{\,\prime}(x)h|}{|h|}.$
╬ЄьхЄшь, ўЄю фы  юЄюсЁрцхэшщ ё ъюэхўэ√ь шёърцхэшхь ш яЁюшчтюы№эюую
$p\geqslant 1$ ъюЁЁхъЄэю юяЁхфхыхэр ш яюўЄш тё■фє ъюэхўэр Єръ
эрч√трхьр  {\it тэєЄЁхээ   фшырЄрЎш  $K_{I, p}(x,f)$ юЄюсЁрцхэш  $f$
яюЁ фър $p$ т Єюўъх $x$}, юяЁхфхы хьр  ЁртхэёЄтрьш
\begin{equation}\label{eq0.1.1A}
K_{I, p}(x,f)\quad =\quad\left\{
\begin{array}{rr}
\frac{|J(x,f)|}{{l\left(f^{\,\prime}(x)\right)}^p}, & J(x,f)\ne 0,\\
1,  &  f^{\,\prime}(x)=0, \\
\infty, & \text{т\,\,юёЄры№э√ї\,\,ёыєўр ї}
\end{array}
\right.\,.
\end{equation}
╧єёЄ№ $\varphi:[0,\infty)\rightarrow[0,\infty)$ -- эхєс√тр■∙р 
ЇєэъЎш , $f\in W_{loc}^{1,1}.$ ┴єфхь уютюЁшЄ№, ўЄю $f:D\rightarrow
{\Bbb R}^n$ яЁшэрфыхцшЄ ъырёёє $W^{1,\varphi}_{loc},$ яш°хь $f\in
W^{1,\varphi}_{loc},$ хёыш $\int\limits_{G}\varphi\left(|\nabla
f(x)|\right)\,dm(x)<\infty$ фы  ы■сющ ъюьяръЄэющ яюфюсырёЄш
$G\subset D,$ уфх $|\nabla
f(x)|=\sqrt{\sum\limits_{i=1}^n\sum\limits_{j=1}^n\left(\frac{\partial
f_i}{\partial x_j}\right)^2}.$ ╩ырёё $W^{1,\varphi}_{loc}$
эрч√трхЄё  ъырёёюь {\it ╬Ёышўр--╤юсюыхтр}. ╨рёёьюЄЁшь ёыхфє■∙є■
чрфрўє:

\medskip
{\it яєёЄ№ $x_0\in D$ ш $f:D\setminus\{x_0\}\rightarrow {\Bbb R}^n$
-- юЄюсЁрцхэшх ъырёёр $W^{1,\varphi}_{loc}(D\setminus\{x_0\})$ ё
ъюэхўэ√ь шёърцхэшхь, Єюуфр яЁш ъръшї єёыютш ї юЄюсЁрцхэшх $f$ ьюцхЄ
с√Є№ яЁюфюыцхэю яю эхяЁхЁ√тэюёЄш т Єюўъє $x_0 ?$ }

\medskip
╬ЄтхЄ эр ¤ЄюЄ тюяЁюё т ёыєўрх, ъюуфр юЄюсЁрцхэшх $f$  ты хЄё 
уюьхюьюЁЇшчьюь с√ы эрщфхэ эрьш эхёъюы№ъю Ёрэхх (ёь.
\cite[ЄхюЁхьр~5]{KRSS} ш \cite[ЄхюЁхьр~9.3]{MRSY}). ╤ЄЁхь ё№ єёшышЄ№
¤ЄюЄ Ёхчєы№ЄрЄ, т эрёЄю ∙хщ ёЄрЄ№х ь√ ЁрёёьрЄЁштрхь сюыхх °шЁюъшщ
ъырёё чрьъэєЄю-юЄъЁ√Є√ї фшёъЁхЄэ√ї юЄюсЁрцхэшщ. ═шцх сєфхЄ яюърчрэю,
ўЄю фы  єърчрээюую ъырёёр чръы■ўхэшх  ю эхяЁхЁ√тэюь яЁюфюыцхэшш т
шчюышЁютрээє■ Єюўъє уЁрэшЎ√ Єръцх тхЁэю, яю ъЁрщэхщ ьхЁх, т ёыєўрх
т√яюыэхэш  ёыхфє■∙хую фюяюыэшЄхы№эюую єёыютш : $C(f, x_0)\cap C(f,
\partial D)=\varnothing.$ ╨рчєьххЄё , яЁюшчтюы№э√х уюьхюьюЁЇшчь√
єфютыхЄтюЁ ■Є ЄЁхсютрэш ь чрьъэєЄюёЄш, фшёъЁхЄэюёЄш, юЄъЁ√ЄюёЄш, р
Єръцх єърчрээюьє юуЁрэшўхэш■ эр яЁхфхы№э√х ьэюцхёЄтр. ╤ фЁєующ
ёЄюЁюэ√, ыхуъю єърчрЄ№ яЁшьхЁ√ эхуюьхюьюЁЇэ√ї чрьъэєЄю-юЄъЁ√Є√ї
фшёъЁхЄэ√ї юЄюсЁрцхэшщ, фы  ъюЄюЁ√ї Єръцх $C(f, x_0)\cap C(f,
\partial D)=\varnothing.$ ╥ръют√ь, эряЁшьхЁ,  ты хЄё  юЄюсЁрцхэшх ё
юуЁрэшўхээ√ь шёърцхэшхь, эрч√трхьюх <<чръЁєўштрэшхь тюъЁєу юёш>> ш
чрфртрхьюх т ЎшышэфЁшўхёъшї ъююЁфшэрЄрї т тшфх $f_m(x)=(r\cos
m\varphi, r\sin m\varphi, x_3,\ldots, x_n),$ $x=(x_1,\ldots, x_n)\in
{\Bbb B}^n,$ $r=|z|,$ $\varphi=\arg z,$ $z=x_1+ix_2,$ $m\in {\Bbb
N}.$ (╟фхё№ $x_0=0$). ═х ыш°эшь сєфхЄ юЄьхЄшЄ№, ўЄю т яЁюшчтюы№эющ
ьхэ№°хщ юсырёЄш єърчрээюх юЄюсЁрцхэшх $f_m$ яЁш эхъюЄюЁюь $m$ єцх эх
чрьъэєЄю. ╤ърцхь, ¤Єю юЄэюёшЄё  ъ юсырёЄш $G:=B(e_1/2, 1/2)\subset
{\Bbb B}^n,$ $e_1=(1,0,\ldots,0),$ уфх єёыютшх $C(f, z_0)\cap C(f,
\partial G)=\varnothing$ Єръцх ьюцхЄ эрЁє°рЄ№ё  фы  эхъюЄюЁющ Єюўъш $z_m\in
G$ ш сюы№°шї $m.$ ─Ёєующ яЁюёЄющ яЁшьхЁ эхуюьхюьюЁЇэюую
чрьъэєЄю-юЄъЁ√Єюую фшёъЁхЄэюую юЄюсЁрцхэш , фы  ъюЄюЁюую юуЁрэшўхэшх
$C(f, x_0)\cap C(f,
\partial D)=\varnothing$ т√яюыэ хЄё , ьюцхЄ с√Є№ фрэ т тшфх $f(z)=z^n,$ $z\in
{\Bbb B}^2\subset {\Bbb C},$ уфх $x_0:=0.$

\medskip
╤ЇюЁьєышЁєхь уыртэ√щ Ёхчєы№ЄрЄ эрёЄю ∙хщ чрьхЄъш.

\medskip
\begin{theorem}\label{th1}
{\, ╧єёЄ№ $n\geqslant 3,$ $D$ -- юуЁрэшўхээр  юсырёЄ№ т ${\Bbb
R}^n,$ $n-1<\alpha\leqslant n,$ $x_0\in D,$ Єюуфр ърцфюх юЄъЁ√Єюх,
фшёъЁхЄэюх ш чрьъэєЄюх юуЁрэшўхээюх юЄюсЁрцхэшх
$f:D\setminus\{x_0\}\rightarrow {\Bbb R}^n$ ъырёёр $W_{loc}^{1,
\varphi}(D\setminus\{x_0\})$ ё ъюэхўэ√ь шёърцхэшхь Єръюх, ўЄю $C(f,
x_0)\cap C(f,
\partial D)=\varnothing,$ яЁюфюыцрхЄё  т Єюўъє $x_0$
эхяЁхЁ√тэ√ь юсЁрчюь фю юЄюсЁрцхэш  $f:D\rightarrow {\Bbb R}^n,$ хёыш
\begin{equation}\label{eqOS3.0a}
\int\limits_{1}^{\infty}\left(\frac{t}{\varphi(t)}\right)^
{\frac{1}{n-2}}dt<\infty
\end{equation}
ш, ъЁюьх Єюую, эрщф╕Єё  ЇєэъЎш  $Q\in L_{loc}^1(D),$ Єрър  ўЄю
$K_{I,\alpha}(x, f)\leqslant Q(x)$ яЁш яюўЄш тёхї $x\in D$ ш яЁш
эхъюЄюЁюь $\varepsilon_0>0,$ $\varepsilon_0<{\rm dist}(x_0,
\partial D),$ т√яюыэхэю ёыхфє■∙хх єёыютшх ЁрёїюфшьюёЄш шэЄхуЁрыр:
\begin{equation}\label{eq9}
\int\limits_{0}^{\varepsilon_0}
\frac{dt}{t^{\frac{n-1}{\alpha-1}}q_{x_0}^{\,\frac{1}{\alpha-1}}(t)}=\infty\,.
\end{equation}
╟фхё№
$q_{x_0}(r):=\frac{1}{\omega_{n-1}r^{n-1}}\int\limits_{|x-x_0|=r}Q(x)\,d{\mathcal
H}^{n-1}$ юсючэрўрхЄ ёЁхфэхх шэЄхуЁры№эюх чэрўхэшх ЇєэъЎшш $Q$ эрф
ёЇхЁющ $S(x_0, r).$ ┬ ўрёЄэюёЄш, чръы■ўхэшх ЄхюЁхь√ \ref{th1}
 ты хЄё  тхЁэ√ь, хёыш $q_{x_0}(r)=\,O\left({\left(
\log{\frac{1}{r}}\right)}^{n-1}\right)$ яЁш $r\rightarrow 0.$}
\end{theorem}

\medskip
\begin{remark}\label{rem2}
╙ёыютшх (\ref{eqOS3.0a}) яЁшэрфыхцшЄ ╩ры№фхЁюэє ш шёяюы№чютрыюё№ фы 
Ёх°хэш  чрфрў эхёъюы№ъю шэюую яырэр (ёь. \cite{Cal}).
\end{remark}

\medskip
╧Ёш $p=n=2$ чръы■ўхэшх ЄхюЁхь√ \ref{th1} ьюцэю эхёъюы№ъю єёшышЄ№.
─ы  ¤Єющ Ўхыш ттхф╕ь ёыхфє■∙шх юсючэрўхэш . ─ы  ъюьяыхъёэючэрўэющ
ЇєэъЎшш $f:D\rightarrow {\Bbb C},$ чрфрээющ т юсырёЄш $D\subset
{\Bbb C},$ шьх■∙хщ ўрёЄэ√х яЁюшчтюфэ√х яю $x$ ш $y$ яЁш яюўЄш тёхї
$z=x + iy,$ яюырурхь $\overline{\partial} f= f_{\overline{z}} =
\left(f_x + if_y\right)/2$ ш $\partial f = f_z = \left(f_x -
if_y\right)/2.$ ╧юырурхь $\mu(z)=\mu_f(z)=f_{\overline{z}}/f_z,$ яЁш
$f_z \ne 0$ ш $\mu(z)=0$ т яЁюЄштэюь ёыєўрх. ╙ърчрээр 
ъюьяыхъёэючэрўэр  ЇєэъЎш  $\mu$ эрч√трхЄё  {\it ъюьяыхъёэющ
фшырЄрЎшхщ} юЄюсЁрцхэш  $f$ т Єюўъх $z.$ {\it ╠ръ\-ёш\-ьры№\-эющ
фшырЄрЎшхщ} юЄюсЁрцхэш  $f$ т Єюўъх $z$ эрч√трхЄё  ёыхфє■∙р 
ЇєэъЎш :
$K_{\mu_f}(z)\quad=\quad K_{\mu}(z)\quad=\quad\frac{1 + |\mu
(z)|}{|1 - |\mu\,(z)||}.$ ╟рьхЄшь, ўЄю $J(f,
z)=|f_z|^2-|f_{\overline{z}}|^2,$ уфх $J(f, z):={\rm
det\,}f^{\,\prime}(z),$
ўЄю ьюцхЄ с√Є№ яЁютхЁхэю яЁ ь√ь яюфёў╕Єюь (ёь., эряЁ.,
\cite[яєэъЄ~C, уы.~I]{A}). ╩Ёюьх Єюую, чрьхЄшь, ўЄю $K_I(z,
f)=K_{\mu}(z).$

\medskip
\begin{theorem}\label{th2}
{\, ╧єёЄ№ $z_0\in D\subset {\Bbb C},$ $a, b \in {\Bbb C},$ $a\ne b,$
Єюуфр ърцфюх юЄъЁ√Єюх фшёъЁхЄэюх юЄюсЁрцхэшх
$f:D\setminus\{z_0\}\rightarrow {\Bbb C}\setminus \{a\cup b\}$
ъырёёр $W_{loc}^{1, 1}$ ё ъюэхўэ√ь шёърцхэшхь яЁюфюыцрхЄё  т Єюўъє
$z_0$ эхяЁхЁ√тэ√ь юсЁрчюь фю юЄюсЁрцхэш  $f:D\rightarrow
\overline{{\Bbb C}},$ хёыш эрщф╕Єё  ЇєэъЎш  $Q\in L_{loc}^1(D),$
Єрър  ўЄю $K_{\mu}(z)\leqslant Q(z)$ яЁш яюўЄш тёхї $z\in D$ ш яЁш
эхъюЄюЁюь $\varepsilon_0>0,$ $\varepsilon_0<{\rm dist}(z_0,
\partial D),$ т√яюыэхэю ёыхфє■∙хх єёыютшх ЁрёїюфшьюёЄш шэЄхуЁрыр
(\ref{eq9}), уфх $q_{z_0}(r):=\frac{1}{2\pi
r}\int\limits_{|z-z_0|=r}Q(z)\,d{\mathcal H}^{1}$ -- ёЁхфэхх
шэЄхуЁры№эюх чэрўхэшх ЇєэъЎшш $Q$ эрф юъЁєцэюёЄ№■ $S(z_0, r).$ ┬
ўрёЄэюёЄш, чръы■ўхэшх ЄхюЁхь√ \ref{th2}  ты хЄё  тхЁэ√ь, хёыш
$q_{z_0}(r)=\,O\left(\log{\frac{1}{r}}\right)$ яЁш $r\rightarrow
0.$}
\end{theorem}

\section{┬ёяюьюурЄхы№э√х ётхфхэш , юёэютэ√х ыхьь√ ш фюърчрЄхы№ёЄтю
ЄхюЁхь√ \ref{th1}} ─юърчрЄхы№ёЄтю юёэютэюую Ёхчєы№ЄрЄр ёЄрЄ№ш
юяшЁрхЄё  эр эхъюЄюЁ√щ ряярЁрЄ, ёєЄ№ ъюЄюЁюую шчырурхЄё  эшцх (ёь.,
эряЁ., \cite{MRSY}). ═ряюьэшь эхъюЄюЁ√х юяЁхфхыхэш , ёт чрээ√х ё
яюэ Єшхь яютхЁїэюёЄш, шэЄхуЁрыр яю яютхЁїэюёЄш, р Єръцх ьюфєыхщ
ёхьхщёЄт ъЁшт√ї ш яютхЁїэюёЄхщ.

\medskip ╧єёЄ№ $\omega$ -- юЄъЁ√Єюх ьэюцхёЄтю т $\overline{{\Bbb
R}^k}:={\Bbb R}^k\cup\{\infty\},$ $k=1,\ldots,n-1.$ ═хяЁхЁ√тэюх
юЄюсЁрцхэшх $S:\omega\rightarrow{\Bbb R}^n$ сєфхь эрч√трЄ№ {\it
$k$-ьхЁэющ яютхЁїэюёЄ№■} $S$ т ${\Bbb R}^n.$ ╫шёыю яЁююсЁрчют
$N(y, S)={\rm card}\,S^{-1}(y)={\rm card}\,\{x\in\omega:S(x)=y\},\
y\in{\Bbb R}^n$ сєфхь эрч√трЄ№ {\it ЇєэъЎшхщ ъЁрЄэюёЄш} яютхЁїэюёЄш
$S.$ ─Ёєушьш ёыютрьш, $N(y, S)$ -- ъЁрЄэюёЄ№ эръЁ√Єш  Єюўъш $y$
яютхЁїэюёЄ№■ $S.$ ╧єёЄ№ $\rho:{\Bbb R}^n\rightarrow\overline{{\Bbb
R}^+}$ -- сюЁхыхтёър  ЇєэъЎш , т Єръюь ёыєўрх шэЄхуЁры юЄ ЇєэъЎшш
$\rho$ яю яютхЁїэюёЄш $S$ юяЁхфхы хЄё  ЁртхэёЄтюь:  $\int\limits_S
\rho\,d{\mathcal{A}}:=\int\limits_{{\Bbb R}^n}\rho(y)\,N(y,
S)\,d{\mathcal H}^ky.$
╧єёЄ№ $\Gamma$ -- ёхьхщёЄтю $k$-ьхЁэ√ї яютхЁїэюёЄхщ $S.$ ┴юЁхыхтёъє■
ЇєэъЎш■ $\rho:{\Bbb R}^n\rightarrow\overline{{\Bbb R}^+}$ сєфхь
эрч√трЄ№ {\it фюяєёЄшьющ} фы  ёхьхщёЄтр $\Gamma,$ ёюъЁ. $\rho\in{\rm
adm}\,\Gamma,$ хёыш
\begin{equation}\label{eq8.2.6}\int\limits_S\rho^k\,d{\mathcal{A}}\geqslant 1\end{equation}
фы  ърцфющ яютхЁїэюёЄш $S\in\Gamma.$ ╧єёЄ№ $p\geqslant 1,$ Єюуфр
{\it $p$-ьюфєыхь} ёхьхщёЄтр $\Gamma$ эрчют╕ь тхышўшэє
$$M_p(\Gamma)=\inf\limits_{\rho\in{\rm adm}\,\Gamma}
\int\limits_{{\Bbb R}^n}\rho^p(x)\,dm(x)\,.$$ ╟рьхЄшь, ўЄю
$p$-ьюфєы№ ёхьхщёЄт яютхЁїэюёЄхщ, юяЁхфхы╕ээ√щ Єръшь юсЁрчюь,
яЁхфёЄрты хЄ ёюсющ тэх°э■■ ьхЁє т яЁюёЄЁрэёЄтх тёхї $k$-ьхЁ\-э√ї
яютхЁїэюёЄхщ (ёь. \cite{Fu}). ├ютюЁ Є, ўЄю эхъюЄюЁюх ётющёЄтю $P$
т√яюыэхэю фы  {\it $p$-яюўЄш тёхї яютхЁїэюёЄхщ} юсырёЄш $D,$ хёыш
юэю шьххЄ ьхёЄю фы  тёхї яютхЁїэюёЄхщ, ыхцр∙шї т $D,$ ъЁюьх, с√Є№
ьюцхЄ, эхъюЄюЁюую шї яюфёхьхщёЄтр, $p$-ьюфєы№ ъюЄюЁюую Ёртхэ эєы■.
╧Ёш $p=n$ яЁшёЄртър <<$p$->> т ёыютрї <<$p$-яюўЄш тёхї...>>, ъръ
яЁртшыю, юяєёърхЄё . ┬ ўрёЄэюёЄш, уютюЁ Є, ўЄю эхъюЄюЁюх ётющёЄтю
т√яюыэхэю фы  {\it $p$-яюўЄш тёхї ъЁшт√ї} юсырёЄш $D$, хёыш юэю
шьххЄ ьхёЄю фы  тёхї ъЁшт√ї, ыхцр∙шї т $D$, ъЁюьх, с√Є№ ьюцхЄ,
эхъюЄюЁюую шї яюфёхьхщёЄтр, $p$-ьюфєы№ ъюЄюЁюую Ёртхэ эєы■.

┴єфхь уютюЁшЄ№, ўЄю шчьхЁшьр  яю ╦хсхує ЇєэъЎш  $\rho:{\Bbb
R}^n\rightarrow\overline{{\Bbb R}^+}$ {\it $p$-юсюс∙╕ээю фюяєёЄшьр}
фы  ёхьхщёЄтр $\Gamma$ $k$-ьхЁэ√ї яютхЁїэюёЄхщ $S$ т ${\Bbb R}^n,$
ёюъЁ. $\rho\in{\rm ext}_p\,{\rm adm}\,\Gamma,$ хёыш ёююЄэю°хэшх
(\ref{eq8.2.6}) т√яюыэхэю фы  $p$-яюўЄш тёхї яютхЁїэюёЄхщ $S$
ёхьхщёЄтр $\Gamma.$ {\it ╬сюс∙╕ээ√щ $p$-ьюфєы№} $\overline
M_p(\Gamma)$ ёхьхщёЄтр $\Gamma$ юяЁхфхы хЄё  ЁртхэёЄтюь
$$\overline{M_p}(\Gamma)= \inf\int\limits_{{\Bbb
R}^n}\rho^p(x)\,dm(x)\,,$$
уфх Єюўэр  эшцэ   уЁрэ№ схЁ╕Єё  яю тёхь ЇєэъЎш ь $\rho\in{\rm
ext}_p\,{\rm adm}\,\Gamma.$ ╬ўхтшфэю, ўЄю яЁш ърцфюь
$p\in(0,\infty),$ $k=1,\ldots,n-1,$ ш ърцфюую ёхьхщёЄтр $k$-ьхЁэ√ї
яютхЁїэюёЄхщ $\Gamma$ т ${\Bbb R}^n,$ т√яюыэхэю ЁртхэёЄтю
$\overline{M_p}(\Gamma)=M_p(\Gamma).$

╤ыхфє■∙шщ ъырёё юЄюсЁрцхэшщ яЁхфёЄрты хЄ ёюсющ юсюс∙хэшх
ътрчшъюэЇюЁьэ√ї юЄюсЁрцхэшщ т ёь√ёых ъюы№Ўхтюую юяЁхфхыхэш  яю
├хЁшэує (\cite{Ge$_3$}) ш юЄфхы№эю шёёыхфєхЄё  (ёь., эряЁ.,
\cite[уыртр~9]{MRSY}). ╧єёЄ№ $p\geqslant 1,$ $D$ ш $D^{\,\prime}$ --
чрфрээ√х юсырёЄш т $\overline{{\Bbb R}^n},$ $n\geqslant 2,$
$x_0\in\overline{D}\setminus\{\infty\}$ ш $Q:D\rightarrow(0,\infty)$
-- шчьхЁшьр  яю ╦хсхує ЇєэъЎш . ┴єфхь уютюЁшЄ№, ўЄю $f:D\rightarrow
D^{\,\prime}$ -- {\it эшцэхх $Q$-юЄюсЁрцхэшх т Єюўъх $x_0$
юЄэюёшЄхы№эю $p$-ьюфєы ,} ъръ Єюы№ъю
\begin{equation}\label{eq1A}
M_p(f(\Sigma_{\varepsilon}))\geqslant \inf\limits_{\rho\in{\rm
ext}_p\,{\rm adm}\,\Sigma_{\varepsilon}}\int\limits_{D\cap
A(\varepsilon, r_0, x_0)}\frac{\rho^p(x)}{Q(x)}\,dm(x)
\end{equation}
фы  ърцфюую ъюы№Ўр $A(\varepsilon, r_0, x_0),$ $r_0\in(0,d_0),$
$d_0=\sup\limits_{x\in D}|x-x_0|,$
уфх $\Sigma_{\varepsilon}$ юсючэрўрхЄ ёхьхщёЄтю тёхї яхЁхёхўхэшщ
ёЇхЁ $S(x_0, r)$ ё юсырёЄ№■ $D,$ $r\in (\varepsilon, r_0).$ ╧ЁшьхЁ√
Єръшї юЄюсЁрцхэшщ эхёыюцэю єърчрЄ№ (ёь. ЄхюЁхьє \ref{thOS4.1} эшцх).

╬ЄьхЄшь, ўЄю т√Ёрцхэш  <<яюўЄш тёхї ъЁшт√ї>> ш <<яюўЄш тёхї
яю\-тхЁ\-ї\-эю\-ё\-Єхщ>> т юЄфхы№э√ї ёыєўр ї ьюуєЄ шьхЄ№ фтх
Ёрчышўэ√х шэЄхЁяЁхЄрЎшш (т ўрёЄэюёЄш, хёыш Ёхў№ шф╕Є ю ёхьхщёЄтх
ёЇхЁ, Єю <<яюўЄш тёхї>> ьюцхЄ яюэшьрЄ№ё  ъръ юЄэюёшЄхы№эю ьэюцхёЄтр
чэрўхэшщ $r,$ Єръ ш $p$-ьюфєы  ёхьхщёЄтр ёЇхЁ, ЁрёёьрЄЁштрхьюую ъръ
ўрёЄэ√щ ёыєўрщ ёхьхщёЄтр яютхЁїэюёЄхщ). ╤ыхфє■∙хх єЄтхЁцфхэшх тэюёшЄ
эхъюЄюЁє■  ёэюёЄ№ ьхцфє єърчрээ√ьш шэЄхЁяЁхЄрЎш ьш ш ьюцхЄ с√Є№
єёЄрэютыхэю яюыэюёЄ№■ яю рэрыюушш ё \cite[ыхььр~9.1]{MRSY}.

\medskip
\begin{lemma}\label{lemma8.2.11}{\, ╧єёЄ№ $p\geqslant 1,$ $x_0\in D.$ ┼ёыш эхъюЄюЁюх
ётющёЄтю $P$ шьххЄ ьхёЄю фы  $p$-яюўЄш тёхї ёЇхЁ $D(x_0, r):=S(x_0,
r)\cap D,$ уфх <<яюўЄш тёхї>> яюэшьрхЄё  т ёь√ёых ьюфєы  ёхьхщёЄт
яютхЁїэюёЄхщ, Єю $P$ Єръцх шьххЄ ьхёЄю фы  яюўЄш тёхї ёЇхЁ $D(x_0,
r)$ юЄэюёшЄхы№эю ышэхщэющ ьхЁ√ ╦хсхур яю ярЁрьхЄЁє $r\in {\Bbb R }.$
╬сЁрЄэю, яєёЄ№ $P$ шьххЄ ьхёЄю фы  яюўЄш тёхї ёЇхЁ $D(x_0,
r):=S(x_0, r)\cap D$ юЄэюёшЄхы№эю ышэхщэющ ьхЁ√ ╦хсхур яю $r\in
{\Bbb R},$ Єюуфр $P$ Єръцх шьххЄ ьхёЄю фы  $p$-яюўЄш тёхї
яютхЁїэюёЄхщ $D(x_0, r):=S(x_0, r)\cap D$ т ёь√ёых ьюфєы  ёхьхщёЄт
яютхЁїэюёЄхщ.}\end{lemma}

\medskip
╤ыхфє■∙хх єЄтхЁцфхэшх юсыхуўрхЄ яЁютхЁъє схёъюэхўэющ ёхЁшш
эхЁртхэёЄт т (\ref{eq1A}) ш ьюцхЄ с√Є№ єёЄрэютыхэю рэрыюушўэю
фюърчрЄхы№ёЄтє \cite[ЄхюЁхьр~9.2]{MRSY} (ёь. Єръцх
\cite[ЄхюЁхьр~6.1]{GS}).

\medskip
\begin{lemma}\label{lemma4}{\,
╧єёЄ№ $D,$  $D^{\,\prime}\subset\overline{{\Bbb R}^n},$
$x_0\in\overline{D}\setminus\{\infty\}$ ш $Q:D\rightarrow(0,\infty)$
-- шчьхЁшьр  яю ╦хсхує ЇєэъЎш . ╬ЄюсЁрцхэшх $f:D\rightarrow
D^{\,\prime}$  ты хЄё  эшцэшь $Q$-юЄюсЁрцхэшхь юЄэюёшЄхы№эю
$p$-ьюфєы  т Єюўъх $x_0,$ $p>n-1,$ Єюуфр ш Єюы№ъю Єюуфр, ъюуфр
%
$M_p(f(\Sigma_{\varepsilon}))\geqslant\int\limits_{\varepsilon}^{r_0}
\frac{dr}{||\,Q||_{s}(r)}\quad\forall\ \varepsilon\in(0,r_0)\,,\
r_0\in(0,d_0),$ $d_0=\sup\limits_{x\in D}|x-x_0|,$
%
$s=\frac{n-1}{p-n+1},$ уфх, ъръ ш т√°х, $\Sigma_{\varepsilon}$
юсючэрўрхЄ ёхьхщёЄтю тёхї яхЁхёхўхэшщ ёЇхЁ $S(x_0, r)$ ё юсырёЄ№■
$D,$ $r\in (\varepsilon, r_0),$
$ \Vert
Q\Vert_{s}(r)=\left(\int\limits_{D(x_0,r)}Q^{s}(x)\,d{\mathcal{A}}\right)^{\frac{1}{s}}$
-- $L_{s}$-эюЁьр ЇєэъЎшш $Q$ эрф ёЇхЁющ $D(x_0,r)=\{x\in D:
|x-x_0|=r\}=D\cap S(x_0,r)$.}
\end{lemma}

\medskip
╧єёЄ№ $G$ -- юЄъЁ√Єюх ьэюцхёЄтю т ${\Bbb R}^n$ ш $I=\{x\in{\Bbb
R}^n:a_i<x_i<b_i,i=1,\ldots,n\}$ -- юЄъЁ√Є√щ $n$-ьхЁэ√щ шэЄхЁтры.
╬ЄюсЁрцхэшх $f:I\rightarrow{\Bbb R}^n$ {\it яЁшэрфыхцшЄ ъырёёє
$ACL$} ({\it рсёюы■Єэю эхяЁхЁ√тэю эр ышэш ї}), хёыш $f$ рсёюы■Єэю
эхяЁхЁ√тэю эр яюўЄш тёхї ышэхщэ√ї ёхуьхэЄрї т $I,$ ярЁрыыхы№э√ї
ъююЁфшэрЄэ√ь юё ь. ╬ЄюсЁрцхэшх $f:G\rightarrow{\Bbb R}^n$ {\it
яЁшэрфыхцшЄ ъырёёє $ACL$} т $G,$ ъюуфр ёєцхэшх $f|_I$ яЁшэрфыхцшЄ
ъырёёє $ACL$ фы  ърцфюую шэЄхЁтрыр $I,$ $\overline{I}\subset G.$

\medskip
═ряюьэшь, ўЄю {\it ъюэфхэёрЄюЁюь} эрч√тр■Є ярЁє
$E=\left(A,\,C\right),$ уфх $A$ -- юЄъЁ√Єюх ьэюцхёЄтю т ${\Bbb
R}^n,$ р $C$ -- ъюьяръЄэюх яюфьэюцхёЄтю $A.$ {\it иьъюёЄ№■}
ъюэфхэёрЄюЁр $E$ яюЁ фър $p\geqslant 1$ эрч√трхЄё  ёыхфє■∙р 
тхышўшэр:
%
%
${\rm cap}_p\,E={\rm cap}_p\,\left(A,\,C\right)= \inf\limits_{u\in
W_0(E)}\,\,\int\limits_A |\nabla u(x)|^p\,\,dm(x),$
%
уфх $W_0(E)=W_0\left(A,\,C\right)$ -- ёхьхщёЄтю эхюЄЁшЎрЄхы№э√ї
эхяЁхЁ√тэ√ї ЇєэъЎшщ $u:A\rightarrow{\Bbb R}$ ё ъюьяръЄэ√ь эюёшЄхыхь
т $A,$ Єръшї ўЄю $u(x)\geqslant 1$ яЁш $x\in C$ ш $u\in ACL.$
╟фхё№, ъръ юс√ўэю, $|\nabla
u|={\left(\sum\limits_{i=1}^n\,{\left(\partial_i u\right)}^2
\right)}^{1/2}.$ ╤ыхфє■∙хх єЄтхЁцфхэшх шьххЄ трцэюх чэрўхэшх фы 
фюърчрЄхы№ёЄтр ьэюушї Ёхчєы№ЄрЄют эрёЄю ∙хщ ЁрсюЄ√ (ёь.
\cite[яЁхфыюцхэшх~10.2, уы.~II]{Ri}).

\medskip
\begin{proposition}\label{pr1*!}{\, ╧єёЄ№ $E=(A,\,C)$ --
яЁюшчтюы№э√щ ъюэфхэёрЄюЁ т ${\Bbb R}^n$ ш яєёЄ№ $\Gamma_E$ --
ёхьхщёЄтю тёхї ъЁшт√ї тшфр $\gamma:[a,\,b)\rightarrow A$ Єръшї, ўЄю
$\gamma(a)\in C$ ш $|\gamma|\cap\left(A\setminus
F\right)\ne\varnothing$ фы  яЁюшчтюы№эюую ъюьяръЄр $F\subset A.$
╥юуфр
${\rm cap}_p\,E=M_p(\Gamma_E).$
}
\end{proposition}

\medskip
╤ыхфє■∙шх трцэ√х ётхфхэш , ърёр■∙шхё  ╕ьъюёЄш ярЁ√ ьэюцхёЄт
юЄэюёшЄхы№эю юсырёЄш, ьюуєЄ с√Є№ эрщфхэ√ т ЁрсюЄх ┬.~╓шьхЁр
\cite{Zi}. ╧єёЄ№ $G$ -- юуЁрэшўхээр  юсырёЄ№ т ${\Bbb R}^n$ ш $C_{0}
, C_{1}$ -- эхяхЁхёхър■∙шхё  ъюьяръЄэ√х ьэюцхёЄтр, ыхцр∙шх т
чрь√ърэшш $G.$ ╧юырурхь  $R=G \setminus (C_{0} \cup C_{1})$ ш
$R^{\,*}=R \cup C_{0}\cup C_{1},$ Єюуфр {\it $p$-╕ьъюёЄ№■ ярЁ√
$C_{0}, C_{1}$ юЄэюёшЄхы№эю чрь√ърэш  $G$} эрч√трхЄё  тхышўшэр
$C_p[G, C_{0}, C_{1}] = \inf \int\limits_{R} \vert \nabla u
\vert^{p}\ dm(x),$
уфх Єюўэр  эшцэ   уЁрэ№ схЁ╕Єё  яю тёхь ЇєэъЎш ь $u,$ эхяЁхЁ√тэ√ь т
$R^{\,*},$ $u\in ACL(R),$ Єръшь ўЄю $u=1$ эр $C_{1}$ ш $u=0$ эр
$C_{0}.$ ╙ърчрээ√х ЇєэъЎшш сєфхь эрч√трЄ№ {\it фюяєёЄшь√ьш} фы 
тхышўшэ√ $C_p [G, C_{0}, C_{1}].$ ╠√ сєфхь уютюЁшЄ№, ўЄю  {\it
ьэюцхёЄтю $\sigma \subset {\Bbb R}^n$ Ёрчфхы хЄ $C_{0}$ ш $C_{1}$ т
$R^{\,*}$}, хёыш $\sigma \cap R$ чрьъэєЄю т $R$ ш эрщфєЄё 
эхяхЁхёхър■∙шхё  ьэюцхёЄтр $A$ ш $B,$  ты ■∙шхё  юЄъЁ√Є√ьш т
$R^{\,*} \setminus \sigma,$ Єръшх ўЄю $R^{\,*} \setminus \sigma =
A\cup B,$ $C_{0}\subset A$ ш $C_{1} \subset B.$ ╧єёЄ№ $\Sigma$
юсючэрўрхЄ ъырёё тёхї ьэюцхёЄт, Ёрчфхы ■∙шї $C_{0}$ ш $C_{1}$ т
$R^{\,*}.$ ─ы  ўшёыр $p^{\prime} = p/(p-1)$ юяЁхфхышь тхышўшэє
%
$$\widetilde{M_{p^{\prime}}}(\Sigma)=\inf\limits_{\rho\in
\widetilde{\rm adm} \Sigma} \int\limits_{{\Bbb
R}^n}\rho^{\,p^{\prime}}dm(x)\,,$$
%
уфх чряшё№ $\rho\in \widetilde{\rm adm}\,\Sigma$ ючэрўрхЄ, ўЄю
$\rho$ -- эхюЄЁшЎрЄхы№эр  сюЁхыхтёър  ЇєэъЎш  т ${\Bbb R}^n$ Єрър ,
ўЄю
%
$$\int\limits_{\sigma \cap R}\rho d{\mathcal H}^{n-1} \geqslant
1\quad\forall\, \sigma \in \Sigma\,.$$
%
╟рьхЄшь, ўЄю ёюуырёэю Ёхчєы№ЄрЄр ╓шьхЁр
\begin{equation}\label{eq3}
\widetilde{M_{p^{\,\prime}}}(\Sigma)=C_p[G , C_{0} ,
C_{1}]^{\,-1/(p-1)}\,,
\end{equation}
ёь. \cite[ЄхюЁхьр~3.13]{Zi} яЁш $p=n$ ш \cite[ё.~50]{Zi$_1$} яЁш
$1<p<\infty.$ ╟рьхЄшь Єръцх, ўЄю ёюуырёэю Ёхчєы№ЄрЄр ╪ы√ър
\begin{equation}\label{eq4}
M_p(\Gamma(E, F, D))= C_p[D, E, F]\,,
\end{equation}
ёь. \cite[ЄхюЁхьр~1]{Shl}.

\medskip
═ряюьэшь, ўЄю юЄюсЁрцхэшх $f:X\rightarrow Y$ ьхцфє яЁюёЄЁрэёЄтрьш ё
ьхЁрьш $(X, \Sigma, \mu)$ ш $(Y, \Sigma^{\,\prime}, \mu^{\,\prime})$
юсырфрхЄ {\it $N$-ётющ\-ё\-Є\-тюь} (╦єчшэр), хёыш шч єёыютш 
$\mu(S)=0$ ёыхфєхЄ, ўЄю $\mu^{\,\prime}(f(S))=0.$ ╤ыхфє■∙хх
тёяюьюурЄхы№эюх єЄтхЁцфхэшх яюыєўхэю т ЁрсюЄх \cite{KRSS} (ёь.
ЄхюЁхьр 1 ш ёыхфёЄтшх 2).

\medskip
\begin{proposition}\label{pr1}
{\, ╧єёЄ№ $D$ -- юсырёЄ№ т ${\Bbb R}^n,$ $n\geqslant 3,$
$\varphi:(0,\infty)\rightarrow (0,\infty)$ -- эхєс√тр■∙р  ЇєэъЎш ,
єфютыхЄтюЁ ■∙р  єёыютш■ (\ref{eqOS3.0a}). ╥юуфр:

1) ┼ёыш $f:D\rightarrow{\Bbb R}^n$ -- эхяЁхЁ√тэюх юЄъЁ√Єюх
юЄюсЁрцхэшх ъырёёр $W^{1,\varphi}_{loc}(D),$ Єю $f$ шьххЄ яюўЄш
тё■фє яюыэ√щ фшЇЇхЁхэЎшры т $D;$

2) ╦■сюх эхяЁхЁ√тэюх юЄюсЁрцхэшх $f\in W^{1,\varphi}_{loc}$ юсырфрхЄ
$N$-ётющёЄтюь юЄэюёшЄхы№эю $(n-1)$-ьхЁэющ ьхЁ√ ╒рєёфюЁЇр, сюыхх
Єюую, ыюъры№эю рсёюы■Єэю эхяЁхЁ√тэю эр яюўЄш тёхї ёЇхЁрї $S(x_0, r)$
ё ЎхэЄЁюь т чрфрээющ яЁхфяшёрээющ Єюўъх $x_0\in{\Bbb R}^n$. ╩Ёюьх
Єюую, эр яюўЄш тёхї Єръшї ёЇхЁрї $S(x_0, r)$ т√яюыэхэю єёыютшх
${\mathcal H}^{n-1}(f(E))=0,$ ъръ Єюы№ъю $|\nabla f|=0$ эр ьэюцхёЄтх
$E\subset S(x_0, r).$ (╟фхё№ <<яюўЄш тёхї>> яюэшьрхЄё  юЄэюёшЄхы№эю
ышэхщэющ ьхЁ√ ╦хсхур яю ярЁрьхЄЁє $r$).}

\end{proposition}

\medskip
─ы  юЄюсЁрцхэш  $f:D\,\rightarrow\,{\Bbb R}^n,$ ьэюцхёЄтр $E\subset
D$ ш $y\,\in\,{\Bbb R}^n,$  юяЁхфхышь {\it ЇєэъЎш■ ъЁрЄэюёЄш $N(y,
f, E)$} ъръ ўшёыю яЁююсЁрчют Єюўъш $y$ тю ьэюцхёЄтх $E,$ Є.х.
\begin{equation}\label{eq1.7A}
N(y, f, E)\,=\,{\rm card}\,\left\{x\in E: f(x)=y\right\}\,,\quad
%
N(f, E)\,=\,\sup\limits_{y\in{\Bbb R}^n}\,N(y, f, E)\,.
\end{equation}
╬сючэрўшь ўхЁхч $J_{n-1}f(a)$ тхышўшэє, ючэрўр■∙є■ $(n-1)$-ьхЁэ√щ
 ъюсшрэ юЄюсЁрцхэш  $f$ т Єюўъх $a$ (ёь. \cite[Ёрчфхы~3.2.1]{Fe}).
╧Ёхфяюыюцшь, ўЄю юЄюсЁрцхэшх $f:D\rightarrow {\Bbb R}^n$
фшЇЇхЁхэЎшЁєхью т Єюўъх $x_0\in D$ ш ьрЄЁшЎр ▀ъюсш
$f^{\,\prime}(x_0)$ эхт√Ёюцфхэр, $J(x_0, f)={\rm
det\,}f^{\,\prime}(x_0)\ne 0.$ ╥юуфр эрщфєЄё  ёшёЄхь√ тхъЄюЁют
$e_1,\ldots, e_n$ ш $\widetilde{e_1},\ldots,\widetilde{e_n}$ ш
яюыюцшЄхы№э√х ўшёыр $\lambda_1(x_0),\ldots,\lambda_n(x_0),$
$\lambda_1(x_0)\leqslant\ldots\leqslant\lambda_n(x_0),$ Єръшх ўЄю
$f^{\,\prime}(x_0)e_i=\lambda_i(x_0)\widetilde{e_i}$ (ёь.
\cite[ЄхюЁхьр~2.1 уы. I]{Re}), яЁш ¤Єюь,
\begin{equation}\label{eq11C}
|J(x_0, f)|=\lambda_1(x_0)\ldots\lambda_n(x_0),\quad \Vert
f^{\,\prime}(x_0)\Vert =\lambda_n(x_0)\,, \quad
l(f^{\,\prime}(x))=\lambda_1(x_0)\,,\end{equation}
\begin{equation}\label{eq41}
K_{I, p}(x_0,
f)=\frac{\lambda_1(x_0)\cdots\lambda_n(x_0)}{\lambda^p_1(x_0)}\,,
\end{equation}
ёь. \cite[ёююЄэю°хэшх~(2.5), Ёрчф.~2.1, уы.~I]{Re}. ╫шёыр
$\lambda_1(x_0),\ldots\lambda_n(x_0)$ эрч√тр■Єё  {\it уыртэ√ьш
чэрўхэш ьш}, р тхъЄюЁр $e_1,\ldots, e_n$ ш
$\widetilde{e_1},\ldots,\widetilde{e_n}$ -- {\it уыртэ√ьш тхъЄюЁрьш
} юЄюсЁрцхэш  $f^{\,\prime}(x_0).$ ╚ч ухюьхЄЁшўхёъюую ёь√ёыр
$(n-1)$-ьхЁэюую  ъюсшрэр, р Єръцх яхЁтюую ёююЄэю°хэш  т
(\ref{eq11C}) т√ЄхърхЄ, ўЄю
\begin{equation}\label{eq10C}
\lambda_1(x_0)\cdots\lambda_{n-1}(x_0)\leqslant
J_{n-1}f(x_0)\leqslant \lambda_2(x_0)\cdots\lambda_n(x_0)\,,
\end{equation}
т ўрёЄэюёЄш, шч (\ref{eq10C}) ёыхфєхЄ, ўЄю $J_{n-1}f(x_0)$
яюыюцшЄхыхэ тю тёхї Єхї Єюўърї $x_0,$ уфх яюыюцшЄхыхэ  ъюсшрэ
$J(x_0, f).$

\medskip
╤ыхфє■∙шх фтх ыхьь√ эхёєЄ т ёхсх юёэютэє■ ёь√ёыютє■ эруЁєчъє фрээющ
чрьхЄъш. ╧хЁтюх шч эшї тяхЁт√х єёЄрэютыхэю фы  ёыєўр  уюьхюьюЁЇшчьют
т ЁрсюЄх \cite{KR$_1$} (ёь. ЄхюЁхьє 2.1).

\medskip
\begin{lemma}{}\label{thOS4.1} { ╧єёЄ№ $D$ -- юсырёЄ№ т ${\Bbb R}^n,$
$n\geqslant 3,$ $\varphi:(0,\infty)\rightarrow (0,\infty)$ --
эхєс√тр■∙р  ЇєэъЎш , єфютыхЄтюЁ ■∙р  єёыютш■ (\ref{eqOS3.0a}).
┼ёыш $n\geqslant 3$ ш $p>n-1,$ Єю ърцфюх юЄъЁ√Єюх фшёъЁхЄэюх
юЄюсЁрцхэшх $f:D\rightarrow {\Bbb R}^n$ ё ъюэхўэ√ь шёърцхэшхь ъырёёр
$W^{1,\varphi}_{loc}$ Єръюх, ўЄю $N(f, D)<\infty,$  ты хЄё  эшцэшь
$Q$-юЄюсЁрцхэшхь юЄэюёшЄхы№эю $p$-ьюфєы  т ърцфющ Єюўъх
$x_0\in\overline{D}$ яЁш
$$Q(x)=N(f, D)\cdot K^{\frac{p-n+1}{n-1}}_{I, \alpha}(x, f),$$
$\alpha:=\frac{p}{p-n+1},$ уфх тэєЄЁхээ   фшырЄрЎш  $K_{I,\alpha}(x,
f)$ юЄюсЁрцхэш  $f$ т Єюўъх $x$ яюЁ фър $\alpha$ юяЁхфхыхэр
ёююЄэю°хэшхь (\ref{eq0.1.1A}), р ъЁрЄэюёЄ№ $N(f, D)$ юяЁхфхыхэр
тЄюЁ√ь ёююЄэю°хэшхь т (\ref{eq1.7A}).}
\end{lemma}

\medskip
\begin{proof}
╟рьхЄшь, ўЄю $f$ фшЇЇхЁхэЎшЁєхью яюўЄш тё■фє ттшфє яЁхфыюцхэш 
\ref{pr1}. ╧єёЄ№ $B$ -- сюЁхыхтю ьэюцхёЄтю тёхї Єюўхъ $x\in D,$ т
ъюЄюЁ√ї $f$ шьххЄ яюыэ√щ фшЇЇхЁхэЎшры $f^{\,\prime}(x)$ ш $J(x,
f)\ne 0.$ ╧Ёшьхэ   ЄхюЁхьє ╩шЁёсЁрєэр ш ётющёЄтю хфшэёЄтхээюёЄш
ряяЁюъёшьрЄштэюую фшЇЇхЁхэЎшрыр (ёь. \cite[яєэъЄ√~2.10.43 ш
3.1.2]{Fe}), ь√ тшфшь, ўЄю ьэюцхёЄтю $B$ яЁхфёЄрты хЄ ёюсющ эх сюыхх
ўхь ёў╕Єэюх юс·хфшэхэшх сюЁхыхтёъшї ьэюцхёЄт $B_l,$
$l=1,2,\ldots\,,$ Єръшї, ўЄю ёєцхэш  $f_l=f|_{B_l}$  ты ■Єё 
сшышя°шЎхт√ьш уюьхюьюЁЇшчьрьш (ёь., эряЁ., \cite[яєэъЄ√~3.2.2, 3.1.4
ш 3.1.8]{Fe}). ┴хч юуЁрэшўхэш  юс∙эюёЄш, ь√ ьюцхь яюырурЄ№, ўЄю
ьэюцхёЄтр $B_l$ яюярЁэю эх яхЁхёхър■Єё . ╬сючэрўшь Єръцх ёшьтюыюь
$B_*$ ьэюцхёЄтю тёхї Єюўхъ $x\in D,$ т ъюЄюЁ√ї $f$ шьххЄ яюыэ√щ
фшЇЇхЁхэЎшры, юфэръю, $f^{\,\prime}(x)=0.$

\medskip
┬тшфє яюёЄЁюхэш , ьэюцхёЄтю $B_0:=D\setminus \left(B\bigcup
B_*\right)$ шьххЄ ыхсхуютє ьхЁє эєы№. ╤ыхфютрЄхы№эю, яю
\cite[ЄхюЁхьр~9.1]{MRSY}, ${\mathcal H}^{n-1}(B_0\cap S_r)=0$ фы 
$p$-яюўЄш тёхї ёЇхЁ $S_r:=S(x_0,r)$ ё ЎхэЄЁюь т Єюўъх
$x_0\in\overline{D},$ уфх <<$p$-яюўЄш тёхї>> ёыхфєхЄ яюэшьрЄ№ т
ёь√ёых $p$-ьюфєы  ёхьхщёЄт яютхЁїэюёЄхщ. ╧ю ыхььх \ref{lemma8.2.11}
Єръцх ${\mathcal H}^{n-1}(B_0\cap S_r)=0$ яЁш яюўЄш тёхї $r\in {\Bbb
R}.$

\medskip
╧ю яЁхфыюцхэш■ \ref{pr1} ш шч єёыютш  ${\mathcal H}^{n-1}(B_0\cap
S_r)=0$ фы  яюўЄш тёхї $r\in {\Bbb R}$ т√ЄхърхЄ, ўЄю ${\mathcal
H}^{n-1}(f(B_0\cap S_r))=0$ фы  яюўЄш тёхї $r\in {\Bbb R}.$ ╧ю ¤Єюьє
яЁхфыюцхэш■ Єръцх ${\mathcal H}^{n-1}(f(B_*\cap S_r))=0,$ яюёъюы№ъє
$f$ -- юЄюсЁрцхэшх ё ъюэхўэ√ь шёърцхэшхь ш, чэрўшЄ, $\nabla f=0$
яюўЄш тё■фє, уфх $J(x, f)=0.$

\medskip
╧єёЄ№ $\Gamma$ -- ёхьхщёЄтю тёхї яхЁхёхўхэшщ ёЇхЁ $S_r,$
$r\in(\varepsilon, r_0),$ $r_0<d_0=\sup\limits_{x\in D}\,|x-x_0|,$ ё
юсырёЄ№■ $D$ (чфхё№ $\varepsilon$ -- яЁюшчтюы№эюх ЇшъёшЁютрээюх
ўшёыю шч шэЄхЁтрыр $(0, r_0)$). ─ы  чрфрээющ ЇєэъЎшш $\rho_*\in{\rm
adm}\,f(\Gamma),$ $\rho_*\equiv0$ тэх $f(D),$ яюырурхь $\rho\equiv
0$ тэх $B,$
$$\rho(x)\ \colon=\ \rho_*(f(x))\left(\frac{|J(x, f)|}{l(f^{\,\prime}(x))}
\right)^{\frac{1}{n-1}} \qquad\text{яЁш}\ x\in B\,.$$
╙ўшЄ√тр  ёююЄэю°хэш  (\ref{eq11C}) ш (\ref{eq10C}),
\begin{equation}\label{eq12C}
\frac{|J(x, f)|}{l(f^{\,\prime}(x))} \geqslant J_{n-1}f(x)\,.
\end{equation}
╧єёЄ№ $D_{r}^{\,*}\in f(\Gamma),$ $D_{r}^{\,*}=f(D\cap S_r).$
╟рьхЄшь, ўЄю
$D_{r}^{\,*}=\bigcup\limits_{i=0}^{\infty} f(S_r\cap B_i)\bigcup
f(S_r\cap B_*)$
ш, ёыхфютрЄхы№эю, фы  яюўЄш тёхї $r\in (\varepsilon, r_0)$
\begin{equation}\label{eq10B}
1\leqslant \int\limits_{D^{\,*}_r}\rho^{n-1}_*(y)d{\mathcal A_*}
\leqslant \sum\limits_{i=0}^{\infty} \int \limits_{f(S_r\cap B_i)}
\rho^{n-1}_*(y)N (y, f, S_r\cap B_i)d{\mathcal H}^{n-1}y +
\end{equation}
$$+\int\limits_{f(S_r\cap B_*)} \rho^{n-1}_*(y)N (y, f, S_r\cap B_*
)d{\mathcal H}^{n-1}y\,.$$ ╙ўшЄ√тр  фюърчрээюх т√°х, шч
(\ref{eq10B}) ь√ яюыєўрхь, ўЄю
\begin{equation}\label{eq11B}
1\leqslant \int\limits_{D^{\,*}_r}\rho^{n-1}_*(y)d{\mathcal A_*}
\leqslant \sum\limits_{i=1}^{\infty} \int \limits_{f(S_r\cap B_i)}
\rho^{n-1}_*(y)N (y, f, S_r\cap B_i)d{\mathcal H}^{n-1}y
\end{equation}
фы  яюўЄш тёхї $r\in (\varepsilon, r_0).$
╨рёёєцфр  яюъєёюўэю эр $B_i,$ $i=1,2,\ldots,$ ттшфє \cite[1.7.6 ш
ЄхюЁхьр~3.2.5]{Fe} ш (\ref{eq12C}) ь√ яюыєўрхь, ўЄю
$$\int\limits_{B_i\cap S_r}\rho^{n-1}\,d{\mathcal A}=
\int\limits_{B_i\cap S_r}\rho_*^{n-1}(f(x))\frac{|J(x,
f)|}{l(f^{\,\prime}(x))}\,d{\mathcal A}=$$
$$=\int\limits_{B_i\cap S_r}\rho_*^{n-1}(f(x))\cdot \frac{|J(x,
f)|}{l(f^{\,\prime}(x))J_{n-1}f(x)}\cdot J_{n-1}f(x)\,d{\mathcal
A}\geqslant $$
\begin{equation}\label{eq12B}
\geqslant\int\limits_{B_i\cap S_r}\rho_*^{n-1}(f(x))\cdot
J_{n-1}f(x)\,d{\mathcal A}=\int\limits_{f(B_i\cap
S_r)}\rho_{*}^{n-1}\,N(y, f, S_r\cap B_i)d{\mathcal H}^{n-1}y
\end{equation} фы  яюўЄш тёхї $r\in (\varepsilon, r_0).$
╚ч (\ref{eq11B}) ш (\ref{eq12B}) т√ЄхърхЄ, ўЄю
$\rho\in{\rm{ext\,adm}}\,\Gamma.$

╟рьхэр яхЁхьхээ√ї эр ърцфюь $B_l,$ $l=1,2,\ldots\,,$ (ёь., эряЁ.,
\cite[ЄхюЁхьр~3.2.5]{Fe}) ш ётющёЄтю ёў╕Єэющ рффшЄштэюёЄш шэЄхуЁрыр
яЁштюф Є ъ юЎхэъх
$$\int\limits_{D}\frac{\rho^p(x)}{K^{\frac{p-n+1}{n-1}}_{I,
\alpha}(x, f)}\,dm(x)\leqslant \int\limits_{f(D)}N(f, D)\cdot
\rho^{\,p}_*(y)\, dm(y)\,,$$ $\alpha:=\frac{p}{p-n+1},$ ўЄю ш
чртхЁ°рхЄ фюърчрЄхы№ёЄтю.
\end{proof}

\medskip
\begin{remark}\label{rem1}
╟ръы■ўхэшх ыхьь√ \ref{thOS4.1} яЁш $n=2$ юёЄр╕Єё  ёяЁртхфышт√ь фы 
ъырёёют ╤юсюыхтр $W_{loc}^{1, 1}$ яЁш рэрыюушўэ√ї єёыютш ї, чр
шёъы■ўхэшхь фюяюыэшЄхы№эюую єёыютш  ╩ры№фхЁюэр (\ref{eqOS3.0a}).
╫Єюс√ т ¤Єюь єсхфшЄ№ё , эхюсїюфшью яютЄюЁшЄ№ фюърчрЄхы№ёЄтю ¤Єющ
ыхьь√ яЁш $n=2,$ уфх эхюсїюфшью єўхёЄ№ эрышўшх $N$-ётющёЄтр
єърчрээ√ї юЄюсЁрцхэшщ эр яюўЄш тёхї юъЁєцэюёЄ ї, ўЄю юсхёяхўштрхЄё 
ётющёЄтюь $ACL$ фы  яЁюшчтюы№э√ї ъырёёют ╤юсюыхтр (ёь.
\cite[ЄхюЁхьр~1, я.~1.1.3, $\S$~1.1, уы.~I]{Maz}).
\end{remark}

\medskip
╚ьххЄ ьхёЄю ёыхфє■∙хх єЄтхЁцфхэшх (ёь. \cite[ыхььр~3.11]{MRV$_2$} ш
\cite[ыхььр~2.6, уы.~III]{Ri} яЁш $p=n$ ш \cite[ыхььр~1]{Sev} яЁш
$p\ne n$).

\medskip
\begin{proposition}\label{pr3*!} {\, ╧єёЄ№ $n-1<p\leqslant n,$ $D$ -- юуЁрэшўхээр 
юсырёЄ№ т ${\Bbb R}^n,$ Єюуфр фы  ърцфюую $a>0$ ёє∙хёЄтєхЄ
яюыюцшЄхы№эюх ўшёыю $\delta
> 0$ Єръюх, ўЄю
${\rm cap}_p\,\left(D,\,C\right)\geqslant \delta,$
%
уфх $C$ -- яЁюшчтюы№э√щ ъюэЄшэєєь т $D$ Єръющ ўЄю $d(C)\geqslant
a.$}
\end{proposition}

\medskip
└эрыюу ёыхфє■∙хщ ыхьь√ т ёыєўрх уюьхюьюЁЇшчьют фюърчрэ т ьюэюуЁрЇшш
\cite[ЄхюЁхьр~9.3]{MRSY} (ёь. Єръцх ЁрсюЄє \cite[ЄхюЁхьр~4.1]{KR}).

\medskip
\begin{lemma}\label{lemma1}
{\, ╧єёЄ№ $n\geqslant 2,$ $D$ -- юуЁрэшўхээр  юсырёЄ№ т ${\Bbb
R}^n,$  $n\geqslant p>n-1,$ $x_0\in D$ ш $Q:D\rightarrow (0,
\infty)$ - шчьхЁшьр  яю ╦хсхує ЇєэъЎш  Єрър , ўЄю яЁш эхъюЄюЁюь
$\varepsilon_0>0,$ $\varepsilon_0<{\rm dist}(x_0,
\partial D),$ т√яюыэхэю єёыютшх
\begin{equation}\label{eq9A}
\int\limits_{0}^{\varepsilon_0}
\frac{dt}{t^{\frac{n-1}{\alpha-1}}\widetilde{q}_{x_0}^{\,\frac{1}{\alpha-1}}(t)}=\infty\,,
\end{equation}
уфх $\alpha=\frac{p}{p-n+1},$
$\widetilde{q}_{x_0}(r):=\frac{1}{\omega_{n-1}r^{n-1}}\int\limits_{|x-x_0|=r}Q^{\frac{n-1}{p-n+1}}(x)\,d{\mathcal
H}^{n-1}$ юсючэрўрхЄ ёЁхфэхх шэЄхуЁры№эюх чэрўхэшх ЇєэъЎшш
$Q^{\frac{n-1}{p-n+1}}(x)$ эрф ёЇхЁющ $S(x_0, r).$ ╥юуфр ърцфюх
юуЁрэшўхээюх юЄъЁ√Єюх, фшёъЁхЄэюх ш чрьъэєЄюх т юсырёЄш
$D\setminus\{x_0\}$ эшцэхх $Q$-юЄюсЁрцхэшх
$f:D\setminus\{x_0\}\rightarrow {\Bbb R}^n$ юЄэюёшЄхы№эю $p$-ьюфєы 
яЁюфюыцрхЄё  т Єюўъє $x_0$ эхяЁхЁ√тэ√ь юсЁрчюь фю юЄюсЁрцхэш 
$f:D\rightarrow {\Bbb R}^n,$ хёыш $C(f, x_0)\cap C(f,
\partial D)=\varnothing.$}
\end{lemma}

\medskip
\begin{proof}
═х юуЁрэшўштр  юс∙эюёЄш Ёрёёєцфхэшщ, ьюцэю ёўшЄрЄ№, ўЄю $x_0=0$ ш
$\overline{f(D\setminus\{0\})}\subset {\Bbb B}^n.$ ╧Ёхфяюыюцшь
яЁюЄштэюх, р шьхээю, ўЄю юЄюсЁрцхэшх $f$ эх ьюцхЄ с√Є№ яЁюфюыцхэю яю
эхяЁхЁ√тэюёЄш т Єюўъє $x_0=0.$ ╥юуфр эрщфєЄё  фтх яюёыхфютрЄхы№эюёЄш
$x_j$ ш $x_j^{\,\prime},$ яЁшэрфыхцр∙шх
$D\setminus\left\{0\right\},$ $x_j\rightarrow 0,\quad
x_j^{\,\prime}\rightarrow 0,$ Єръшх, ўЄю
$|f(x_j)-f(x_j^{\,\prime})|\geqslant a>0$ фы  тёхї $j\in {\Bbb N}.$
╠юцэю ёўшЄрЄ№, ўЄю $x_j$ ш $x_j^{\,\prime}$ ыхцрЄ тэєЄЁш °рЁр $B(0,
r_0),$ $r_0:={\rm dist\,}(0, \partial D).$ ╧юырурхь
$r_j=\max{\left\{|x_j|,\,|x_j^{\,\prime}|\right\}},
l_j=\min{\left\{|x_j|,\,|x_j^{\,\prime}|\right\}}.$
╤юхфшэшь Єюўъш $x_j$ ш $x_j^{\,\prime}$ чрьъэєЄющ ъЁштющ, ыхцр∙хщ т
$\overline{B(0, r_j)}\setminus\left\{0\right\}.$ ╬сючэрўшь ¤Єє
ъЁштє■ ёшьтюыюь $C_j$ ш ЁрёёьюЄЁшь ъюэфхэёрЄюЁ
$E_j=\left(D\setminus\left\{0\right\}\,,C_j\right)$ (эх юуЁрэшўштр 
юс∙эюёЄш, ьюцэю ёўшЄрЄ№, ўЄю тёх Єюўъш $x\in C_j$ єфютыхЄтюЁ ■Є
эхЁртхэёЄтє $|x|\geqslant l_j$). ┬ ёшыє юЄъЁ√ЄюёЄш ш эхяЁхЁ√тэюёЄш
юЄюсЁрцхэш  $f,$ ярЁр $f(E_j)$ Єръцх  ты хЄё  ъюэфхэёрЄюЁюь.
╧юёъюы№ъє $f$ -- юЄъЁ√Єюх, фшёъЁхЄэюх ш чрьъэєЄюх юЄюсЁрцхэшх,
$\partial f(D\setminus\{0\})= C(f, \partial D)\cup C(f, 0).$

\medskip
╨рёёьюЄЁшь яЁш $r_j<r<r_0$ яЁюъюыюЄ√щ °рЁ $G_1:=B(0,
r)\setminus\{0\}.$ ╟рьхЄшь, ўЄю $C_j$ -- ъюьяръЄэюх яюфьэюцхёЄтю
$G_1,$ Єюуфр $f(C_j)$ -- ъюьяръЄэюх яюфьэюцхёЄтю $f(G_1).$

\medskip
┬тшфє юЄъЁ√ЄюёЄш $f$ шьххЄ ьхёЄю тъы■ўхэшх $\partial f(G_1)\subset
C(f, 0)\cup f(S(0, r)).$ ─хщёЄтшЄхы№эю, хёыш $y_0\in \partial
f(G_1),$ Єю фы  эхъюЄюЁющ яюёыхфютрЄхы№эюёЄш $y_k\in f(G_1)$ шьххь:
$y_k\rightarrow y_0.$ ╥юуфр $y_k=f(x_k),$ $x_k\in G_1.$ ╧юёъюы№ъє
$G_1$ юуЁрэшўхэю, Єю ьюцэю ёўшЄрЄ№, ўЄю $x_k\rightarrow x_0\in
\overline{G_1}.$ ╬ёЄрыюё№ чрьхЄшЄ№, ўЄю ёыєўрщ, ъюуфр $x_0$ --
тэєЄЁхээ   Єюўър $G_1$ эхтючьюцхэ, яюёъюы№ъє т ¤Єюь ёыєўрх
$f(x_k)\rightarrow f(x_0),$ уфх $f(x_0)$ -- тэєЄЁхээ   Єюўър
$f(G_1),$ ўЄю яЁюЄштюЁхўшЄ т√сюЁє $y_k.$ ╥юуфр $x_0\in \partial
G_1=\{0\}\cup S(0, r),$ ўЄю ш фюърч√трхЄ тъы■ўхэшх $\partial
f(G_1)\subset C(f, 0)\cup f(S(0, r)).$ ╥юуфр ттшфє чрьъэєЄюёЄш ш
юЄъЁ√ЄюёЄш юЄюсЁрцхэш  $f$ ьэюцхёЄтю $\partial f(G_1)\setminus C(f,
0)$  ты хЄё  чрьъэєЄ√ь т ${\Bbb R}^n.$

\medskip
╬Єё■фр т√ЄхърхЄ, ўЄю ьэюцхёЄтю $\sigma:=\partial f(G_1)\setminus
C(f, 0)$ юЄфхы хЄ $f(C_j)$ юЄ $C(f, \partial D)$ т
$f(D\setminus\{0\})\cup C(f, \partial D).$ ─хщёЄтшЄхы№эю,
$$f(D\setminus\{0\})\cup C(f, \partial D)=f(G_1)\cup\sigma\cup \left((f(D\setminus\{0\})\cup C(f, \partial D))\setminus
\overline{f(G_1)}\right)\,,$$
ърцфюх шч ьэюцхёЄт $A:=f(G_1)$ ш $B:=(f(D\setminus\{0\})\cup C(f,
\partial D))\setminus \overline{f(G_1)}$ юЄъЁ√Єю т Єюяюыюушш
яЁюёЄЁрэёЄтр $f(D\setminus\{0\})\cup C(f, \partial D),$ $A\cap
B=\varnothing,$ $C_0:=f(C_j)\subset A$ ш $C_1:=C(f, \partial
D)\subset B.$

\medskip
╧юырурхь $\alpha:=\frac{p}{p-n+1}.$ ╧юёъюы№ъє $\sigma\subset f(S(0,
r)),$ ттшфє (\ref{eq3}) ш (\ref{eq4})
\begin{equation}\label{eq5C}
M_{\alpha}(\Gamma(f(C_j), C(f, \partial D), f(D\setminus
\{0\})))\leqslant
\frac{1}{M_p^{\frac{n-1}{p-n+1}}(f(\Sigma_{r}))}\,,
\end{equation}
уфх $\Sigma_{r}$ -- ёхьхщёЄтю ёЇхЁ $S(0, r),$ $r\in (r_j, r_0).$
╤ фЁєующ ёЄюЁюэ√, шч ыхьь√ \ref{lemma4} ш єёыютш  ЁрёїюфшьюёЄш
шэЄхуЁрыр (\ref{eq9A}) т√ЄхърхЄ, ўЄю
$M_p^{\frac{n-1}{p-n+1}}(f(\Sigma_{r}))\rightarrow\infty$ яЁш
$j\rightarrow \infty.$ ┬ Єръюь ёыєўрх, шч (\ref{eq5C}) ёыхфєхЄ, ўЄю
яЁш $j\rightarrow \infty$
\begin{equation}\label{eq6C}
M_{\alpha}(\Gamma(C(f, D), f(C_j), f(D\setminus\{0\})))\rightarrow
0\,.
\end{equation}
└эрыюушўэє■ яЁюЎхфєЁє яЁюфхырхь юЄэюёшЄхы№эю яЁхфхы№эюую ьэюцхёЄтр
$C(f, 0).$ ╚ьхээю, чрьхЄшь, ўЄю $C_j$ -- ъюьяръЄ т $G_2:=D\setminus
\overline{B(0, \varepsilon)}$ фы  яЁюшчтюы№эюую $\varepsilon\in (0,
l_j).$ ╥юуфр ттшфє эхяЁхЁ√тэюёЄш $f$ ьэюцхёЄтю $f(C_j)$  ты хЄё 
ъюьяръЄэ√ь яюфьэюцхёЄтюь $f(G_2)$ ш, т ўрёЄэюёЄш, $\partial
f(G_2)\cap f(C_j)=\varnothing.$ ─рыхх, чрьхЄшь, ўЄю
\begin{equation}\label{eq13A}
\partial
f(G_2)\subset C(f,
\partial D)\cup f(S(0, \varepsilon))\,.
\end{equation}
╧юырурхь $\theta:=\partial f(G_2)\setminus C(f, \partial D)$ ш
чрьхЄшь, ўЄю $\theta$  ты хЄё  чрьъэєЄ√ь, яюёъюы№ъє шьххЄ ьхёЄю
ёююЄэю°хэшх (\ref{eq13A}) ш, ъЁюьх Єюую, $C(f,
\partial D)\cap f(S(0, \varepsilon))=\varnothing$ ттшфє чрьъэєЄюёЄш
юЄюсЁрцхэш  $f$ т $D\setminus\{0\}.$ ╩Ёюьх Єюую, чрьхЄшь, ўЄю
$\theta$ юЄфхы хЄ $C_3:=f(C_j)$ ш $C_4:=C(f, 0)$ т
$f(D\setminus\{0\})\cup C(f, 0).$ ─хщёЄтшЄхы№эю,
$$f(D\setminus\{0\})\cup C(f, 0)= f(G_2)\cup \theta\cup
\left((f(D\setminus\{0\})\cup C(f,
0))\setminus\overline{f(G_2)}\right)\,,$$
$A=f(G_2)$ ш $B=\left((f(D\setminus\{0\})\cup C(f,
0))\setminus\overline{f(G_2)}\right)$ юЄъЁ√Є√ т Єюяюыюушш
яЁюёЄЁрэёЄтр $f(D\setminus\{0\})\cup C(f, 0),$  $A\cap
B=\varnothing,$ $C_3:=f(C_j)\subset A$ ш $C_4:=C(f, 0)\subset B.$

\medskip
╩ръ ш яЁхцфх, яюырурхь $\alpha:=\frac{p}{p-n+1}.$  ╥ръ ъръ
$\theta\subset f(S(0, \varepsilon)),$ ттшфє (\ref{eq3}) ш
(\ref{eq4}) яюыєўрхь:
\begin{equation}\label{eq7C}
M_{\alpha}(\Gamma(f(C_j), C(f, 0), f(D\setminus \{0\})))\leqslant
\frac{1}{M_p^{\frac{n-1}{p-n+1}}(f(\Theta_{\varepsilon}))}\,,
\end{equation}
уфх $\Theta_{\varepsilon}$ -- ёхьхщёЄтю ёЇхЁ $S(0, \varepsilon),$
$\varepsilon\in (0, l_j).$
╤ фЁєующ ёЄюЁюэ√, шч ыхьь√ \ref{lemma4} ш єёыютш  ЁрёїюфшьюёЄш
шэЄхуЁрыр (\ref{eq9A}) т√ЄхърхЄ, ўЄю
$M_p^{\frac{n-1}{p-n+1}}(f(\Theta_{\varepsilon}))=\infty.$ ┬ Єръюь
ёыєўрх, шч (\ref{eq7C}) ёыхфєхЄ, ўЄю
\begin{equation}\label{eq8C}
M_{\alpha}(\Gamma(C(f, 0), f(C_j), f(D\setminus\{0\})))=0\,.
\end{equation}
╟рьхЄшь, ўЄю ттшфє яЁхфыюцхэш  \ref{pr1*!} ш яюыєрффшЄштэюёЄш ьюфєы 
ёьхщёЄт ъЁшт√ї (ёь. \cite[Ёрчф.~6, уы.~I]{Va}), яЁш $j\rightarrow
\infty$ шч (\ref{eq6C}) ш (\ref{eq8C}) т√ЄхърхЄ, ўЄю
$${\rm cap}_{\alpha}\,f(E_j)\leqslant$$
\begin{equation}\label{eq9C}
\leqslant M_{\alpha}(\Gamma(C(f, 0), f(C_j), f(D\setminus\{0\})))+
M_{\alpha}(\Gamma(C(f,
\partial D), f(C_j), f(D\setminus\{0\})))\rightarrow 0\,.
\end{equation}
╤ фЁєующ ёЄюЁюэ√, чрьхЄшь, ўЄю яЁш ёфхырээ√ї юуЁрэшўхэш ї эр $p,$
тхышўшэр $\alpha$ Єръцх єфютыхЄтюЁ хЄ єёыютш■ $\alpha>n-1.$ ┬ Єръюь
ёыєўрх, яю яЁхфыюцхэш■ \ref{pr3*!} ${\rm
cap}_{\alpha}f(E_j)\geqslant \delta>0$ яЁш тёхї эрЄєЁры№э√ї $j,$ ўЄю
яЁюЄштюЁхўшЄ (\ref{eq9C}). ╦хььр фюърчрэр.
 \end{proof}

 \medskip
{\it ─юърчрЄхы№ёЄтю ЄхюЁхь√ \ref{th1}.} ╧ю ыхььх \ref{thOS4.1}
юЄюсЁрцхэшх $f$ т ърцфющ Єюўъх $x_0\in D$  ты хЄё  эшцэшь
$Q$-юЄюсЁрцхэшхь юЄэюёшЄхы№эю $p$-ьюфєы  т ърцфющ Єюўъх
$x_0\in\overline{D}$ яЁш $Q(x)=N(f, D)\cdot
K^{\frac{p-n+1}{n-1}}_{I, \alpha}(x, f),$ $\alpha:=\frac{p}{p-n+1}$
(Є.х., $p=\frac{\alpha(n-1)}{\alpha-1}$), уфх тэєЄЁхээ   фшырЄрЎш 
$K_{I,\alpha}(x, f)$ юЄюсЁрцхэш  $f$ т Єюўъх $x$ яюЁ фър $\alpha$
юяЁхфхыхэр ёююЄэю°хэшхь (\ref{eq0.1.1A}), р ъЁрЄэюёЄ№ $N(f, D)$
юяЁхфхыхэр тЄюЁ√ь ёююЄэю°хэшхь т (\ref{eq1.7A}). ╟рьхЄшь, ўЄю,
яюёъюы№ъє $\alpha\in (n-1, n],$ Єю Єръцх $p\in (n-1, n].$ ╥юуфр
эхюсїюфшьюх чръы■ўхэшх т√ЄхърхЄ шч ыхьь√ \ref{lemma1}, р Єръцх Єюую
ЇръЄр, ўЄю ьръёшьры№эр  ъЁрЄэюёЄ№ $N(f, D)$ чрьъэєЄюую юЄъЁ√Єюую
фшёъЁхЄэюую юЄюсЁрцхэш  $f$ ъюэхўэр (ёь., эряЁ.,
\cite[ыхььр~3.3]{MS}). $\Box$

\medskip
╥хяхЁ№ юЄфхы№эю шёёыхфєхь ёыєўрщ $n=2.$ ─ы  ¤Єющ Ўхыш эряюьэшь, ўЄю
юЄюсЁрцхэшх $f:D\rightarrow \overline{{\Bbb R}^n}$ эрч√трхЄё  {\it
ъюы№Ўхт√ь $Q$-юЄюс\-Ёр\-цх\-эшхь т Єюўъх $x_0\,\in\,D$} (ёь.
\cite{MRSY}--\cite{MRSY$_1$}), хёыш ёююЄэю°хэшх
%
$M\left(f\left(\Gamma\left(S_1,\,S_2,\,A\right)\right)\right)\
\leqslant
\int\limits_{A} Q(x)\cdot \eta^n(|x-x_0|)\ dm(x)$ 
т√яюыэхэю фы  ы■сюую ъюы№Ўр $A=A(r_1,r_2, x_0),$\, $0<r_1<r_2<
r_0:={\rm dist\,}(x_0, \partial D),$ ш фы  ърцфющ шчьхЁшьющ ЇєэъЎшш
$\eta : (r_1,r_2)\rightarrow [0,\infty ]\,$ Єръющ, ўЄю
%
%
%
$\int\limits_{r_1}^{r_2}\eta(r)dr\geqslant 1.$ ╬ЄьхЄшь, ўЄю
ъюы№Ўхт√х $Q$-уюьхюьюЁЇшчь√ яЁюфюыцр■Єё  яю эхяЁхЁ√тэюёЄш т
шчюышЁютрээ√х уЁрэшўэ√х Єюўъш, яЁшў╕ь яЁюфюыцхээюх юЄюсЁрцхэшх Єръцх
 ты хЄё  уюьхюьюЁЇшчьюь (ёь. \cite[ыхььр~4 ш ЄхюЁхьр~4]{Lom}).

\medskip
{\it ─юърчрЄхы№ёЄтю ЄхюЁхь√ \ref{th2}.} ╧єёЄ№ $f$ -- юЄюсЁрцхэшх шч
єёыютш  ЄхюЁхь√, Єюуфр, т ўрёЄэюёЄш,  $f\in W_{loc}^{1,1},$ $f$ --
ъюэхўэюую шёърцхэш  т $D\setminus\{z_0\},$ ъЁюьх Єюую, $f$ фшёъЁхЄэю
ш юЄъЁ√Єю. ╥юуфр ёюуырёэю яЁхфёЄртыхэш■ ╤Єюшыютр \cite[я.~5 (III),
уы.~V]{St}, $f=\varphi\circ g,$ уфх $g$ -- эхъюЄюЁ√щ уюьхюьюЁЇшчь, р
$\varphi$ -- рэрышЄшўхёър  ЇєэъЎш . ╟рьхЄшь, ўЄю Єюуфр Єръцх $g\in
W_{loc}^{1,1}$ ш, ъЁюьх Єюую, $g$ шьххЄ ъюэхўэюх шёърцхэшх.

─хщёЄтшЄхы№эю, ьэюцхёЄтю Єюўхъ тхЄтыхэш  $B_{\varphi}\subset
g(D\setminus\{z_0\})$ ЇєэъЎшш $\varphi$ ёюёЄюшЄ Єюы№ъю шч
шчюышЁютрээ√ї Єюўхъ (ёь. \cite[яєэъЄ√ 5 ш 6 (II), уы.~V]{St}).
╤ыхфютрЄхы№эю, $g(z)=\varphi^{-1}\circ f$ ыюъры№эю, тэх ьэюцхёЄтр
$g^{-1}\left(B_{\varphi}\right).$ ▀ёэю, ўЄю ьэюцхёЄтю
$g^{-1}\left(B_{\varphi}\right)$ Єръцх ёюёЄюшЄ шч шчюышЁютрээ√ї
Єюўхъ, ёыхфютрЄхы№эю, $g\in ACL(D\setminus\{z_0\})$ ъръ ъюьяючшЎш 
рэрышЄшўхёъющ ЇєэъЎшш $\varphi^{-1}$ ш юЄюсЁрцхэш  $f\in
W_{loc}^{1,1}(D\setminus\{z_0\}).$

╧юърцхь, ўЄю $g\in W_{loc}^{1,1}(D\setminus\{z_0\}).$ ─ы  ¤Єющ Ўхыш,
яюёъюы№ъє $g\in ACL(D\setminus\{z_0\}),$ эрь фюёЄрЄюўэю яюърчрЄ№,
ўЄю $|\partial g|\in L_{loc}^1 (D\setminus\{z_0\})$ ш
$|\overline{\partial} g|\in L_{loc}^1 (D\setminus\{z_0\})$ (ёь.
\cite[ЄхюЁхь√ 1 ш 2, я.~1.1.3]{Maz}).
 ╧єёЄ№ фрыхх $\mu_f(z)$ ючэрўрхЄ ъюьяыхъёэє■ фшырЄрЎш■ ЇєэъЎшш
$f(z),$ р $\mu_g(z)$ -- ъюьяыхъёэє■ фшырЄрЎш■ $g.$ ╤юуырёэю
\cite[(1), я.~C, уы.~I]{A} фы  яюўЄш тёхї $z\in D\setminus\{z_0\}$
яюыєўрхь:
\begin{equation}\label{eq1}
f_z=\varphi_z(g(z))g_z,\qquad
f_{\overline{z}}=\varphi_z(g(z))g_{\overline{z}}\,,
\end{equation}
$\mu_f(z)=\mu_g(z)=:\mu(z), \quad
K_{\mu_f}(z)=K_{\mu_g}(z):=K_{\mu}(z)=\frac{1+|\mu|}{|1-|\mu||}.$
╥ръшь юсЁрчюь,  $K_{\mu}(z)\in L_{loc}^1(D\setminus\{z_0\}).$
╧юёъюы№ъє $f$ -- ъюэхўэюую шёърцхэш , шч (\ref{eq1}) эхьхфыхээю
ёыхфєхЄ, ўЄю $g$ Єръцх ъюэхўэюую шёърцхэш  ш яЁш яюўЄш тёхї $z\in
D\setminus\{z_0\}$ т√яюыэхэ√ ёююЄэю°хэш 
$|\partial g|\leqslant |\partial g|+ |\overline{\partial} g|=
K^{1/2}_{\mu}(z)(|J(f, z)|)^{1/2},$
юЄъєфр яю эхЁртхэёЄтє ├╕ы№фхЁр $|\partial g|\in L_{loc}^1
(D\setminus\{z_0\})$ ш $|\overline{\partial} g|\in L_{loc}^1
(D\setminus\{z_0\}).$ ╤ыхфютрЄхы№эю, $g\in
W_{loc}^{1,1}(D\setminus\{z_0\})$ ш $g$ шьххЄ ъюэхўэюх шёърцхэшх.

┬ Єръюь ёыєўрх, $g$ яЁюфюыцрхЄё  фю уюьхюьюЁЇшчьр $g: D\rightarrow
{\Bbb C}$ ттшфє \cite[ыхььр~4 ш ЄхюЁхьр~4]{Lom}. ╥юуфр $\varphi$
яЁюфюыцрхЄё  яю эхяЁхЁ√тэюёЄш т Єюўъє $g(z_0)$ юсырёЄш $g(D)$ ттшфє
ъырёёшўхёъющ ЄхюЁхь√ ╧шърЁр, ўЄю ш фюърч√трхЄ ЄхюЁхьє. $\Box$

\medskip
\section{═хъюЄюЁ√х ёыхфёЄтш  ш чрьхўрэш } ┼∙╕ юфшэ трцэ√щ Ёхчєы№ЄрЄ,
юЄэюё ∙шщё  ъ єёЄЁрэхэш■ юёюсхээюёЄхщ ъырёёют ╬Ёышўр--╤юсюыхтр,
ърёрхЄё  ЇєэъЎшщ ъюэхўэюую ёЁхфэхую ъюыхсрэш  (ёь. \cite{MRSY} ш
\cite{IR}).

\medskip
┬ фры№эхщ°хь эрь яюэрфюсшЄё  ёыхфє■∙хх тёяюьюурЄхы№эюх єЄтхЁцфхэшх
(ёь., эряЁ., \cite[ыхььр~7.4, уы.~7]{MRSY} ш
\cite[ыхььр~2.2]{RS$_1$}) яЁш $p=n$ ш \cite[ыхььр~2.2]{Sal} яЁш
$p\ne n.$

\medskip
\begin{proposition}\label{pr1A}
{\, ╧єёЄ№  $x_0 \in {\Bbb R}^n,$ $Q(x)$ -- шчьхЁшьр  яю ╦хсхує
ЇєэъЎш , $Q:{\Bbb R}^n\rightarrow [0, \infty],$ $Q\in
L_{loc}^1({\Bbb R}^n).$ ╧юырурхь $A:=A(r_1,r_2,x_0)=\{ x\,\in\,{\Bbb
R}^n : r_1<|x-x_0|<r_2\}$ ш
$\eta_0(r)=\frac{1}{Ir^{\frac{n-1}{p-1}}q_{x_0}^{\frac{1}{p-1}}(r)},$
уфх $I:=I=I(x_0,r_1,r_2)=\int\limits_{r_1}^{r_2}\
\frac{dr}{r^{\frac{n-1}{p-1}}q_{x_0}^{\frac{1}{p-1}}(r)}$ ш
$q_{x_0}(r):=\frac{1}{\omega_{n-1}r^{n-1}}\int\limits_{|x-x_0|=r}Q(x)\,d{\mathcal
H}^{n-1}$ -- ёЁхфэхх шэЄхуЁры№эюх чэрўхэшх ЇєэъЎшш $Q$ эрф ёЇхЁющ
$S(x_0, r).$ ╥юуфр
\begin{equation}\label{eq10A}
\frac{\omega_{n-1}}{I^{p-1}}=\int\limits_{A} Q(x)\cdot
\eta_0^p(|x-x_0|)\ dm(x)\leqslant\int\limits_{A} Q(x)\cdot
\eta^p(|x-x_0|)\ dm(x)
\end{equation}
фы  ы■сющ шчьхЁшьющ яю ╦хсхує ЇєэъЎшш $\eta :(r_1,r_2)\rightarrow
[0,\infty]$ Єръющ, ўЄю
$\int\limits_{r_1}^{r_2}\eta(r)dr=1. $ }
\end{proposition}

\medskip ┴єфхь уютюЁшЄ№, ўЄю ыюъры№эю шэЄхуЁшЁєхьр 
ЇєэъЎш  ${\varphi}:D\rightarrow{\Bbb R}$ шьххЄ {\it ъюэхўэюх ёЁхфэхх
ъюыхсрэшх} т Єюўъх $x_0\in D$, яш°хь $\varphi\in FMO(x_0),$ хёыш
%
%
%
%
$$\limsup\limits_{\varepsilon\rightarrow
0}\frac{1}{\Omega_n\varepsilon^n}\int\limits_{B( x_0,\,\varepsilon)}
|{\varphi}(x)-\overline{{\varphi}}_{\varepsilon}|\,
dm(x)<\infty\,,$$
%
%
уфх
$\overline{{\varphi}}_{\varepsilon}=\frac{1}
{\Omega_n\varepsilon^n}\int\limits_{B(x_0,\,\varepsilon)}
{\varphi}(x)\, dm(x).$
\medskip
╩ръ шчтхёЄэю, $\Omega_n\varepsilon^n=m(B(x_0, \varepsilon)).$ ╚ьххЄ
ьхёЄю ёыхфє■∙р 

\medskip
\begin{theorem}\label{th3}
{\, ╧єёЄ№ $n\geqslant 3,$ $D$ -- юуЁрэшўхээр  юсырёЄ№ т ${\Bbb
R}^n,$  $n-1<\alpha\leqslant n,$ $x_0\in D,$ Єюуфр ърцфюх юЄъЁ√Єюх,
фшёъЁхЄэюх ш чрьъэєЄюх юуЁрэшўхээюх юЄюсЁрцхэшх
$f:D\setminus\{x_0\}\rightarrow {\Bbb R}^n$ ъырёёр $W_{loc}^{1,
\varphi}(D\setminus\{x_0\})$ ё ъюэхўэ√ь шёърцхэшхь Єръюх, ўЄю $C(f,
x_0)\cap C(f,
\partial D)=\varnothing,$ яЁюфюыцрхЄё  т Єюўъє $x_0$
эхяЁхЁ√тэ√ь юсЁрчюь фю юЄюсЁрцхэш  $f:D\rightarrow {\Bbb R}^n,$ хёыш
т√яюыэхэю єёыютшх (\ref{eqOS3.0a}) ш, ъЁюьх Єюую, эрщф╕Єё  ЇєэъЎш 
$Q\in L_{loc}^{1}(D),$ Єрър  ўЄю $K_{I,\alpha}(x, f)\leqslant Q(x)$
яЁш яюўЄш тёхї $x\in D$ ш $Q\in FMO (x_0).$}
\end{theorem}

\medskip
\begin{proof} ─юёЄрЄюўэю яюърчрЄ№, ўЄю
єёыютшх $Q\in FMO(x_0)$ тыхў╕Є ЁрёїюфшьюёЄ№ шэЄхуЁрыр (\ref{eq9}),
яюёъюы№ъє т ¤Єюьє ёыєўрх эхюсїюфшьюх чръы■ўхэшх сєфхЄ ёыхфютрЄ№ шч
ЄхюЁхь√ \ref{th1}. ╟рьхЄшь, ўЄю фы  ЇєэъЎшщ ъырёёр $FMO$  т Єюўъх
$x_0$
\begin{equation}\label{eq31*}
\int\limits_{\varepsilon<|x|<{e_0}}\frac{Q(x+x_0)\, dm(x)}
{\left(|x| \log \frac{1}{|x|}\right)^n} = O \left(\log\log
\frac{1}{\varepsilon}\right)
\end{equation}
яЁш  $\varepsilon \rightarrow 0 $ ш фы  эхъюЄюЁюую $e_0>0,$ $e_0
\leqslant {\rm dist}\,\left(0,\partial D\right).$ ╧Ёш
$\varepsilon_0<r_0:={\rm dist}\,\left(0,\partial D\right)$ яюырурхь
$\psi(t):=\frac{1}{\left(t\,\log{\frac1t}\right)^{n/{\alpha}}},$
$I(\varepsilon,
\varepsilon_0):=\int\limits_{\varepsilon}^{\varepsilon_0}\psi(t) dt
\geqslant \log{\frac{\log{\frac{1}
{\varepsilon}}}{\log{\frac{1}{\varepsilon_0}}}}$ ш
$\eta(t):=\psi(t)/I(\varepsilon, \varepsilon_0).$ ╟рьхЄшь, ўЄю
$\int\limits_{\varepsilon}^{\varepsilon_0}\eta(t)dt=1,$ ъЁюьх Єюую,
шч ёююЄэю°хэш  (\ref{eq31*}) т√ЄхърхЄ, ўЄю
\begin{equation}\label{eq32*}
\frac{1}{I^{\alpha}(\varepsilon,
\varepsilon_0)}\int\limits_{\varepsilon<|x|<\varepsilon_0}
Q(x+x_0)\cdot\psi^{\alpha}(|x|)
 \ dm(x)\leqslant C\left(\log\log\frac{1}{\varepsilon}\right)^{1-{\alpha}}\rightarrow
 0
 \end{equation}
яЁш $\varepsilon\rightarrow 0.$ ╚ч ёююЄэю°хэшщ (\ref{eq10A}) ш
(\ref{eq32*}) т√ЄхърхЄ, ўЄю шэЄхуЁры тшфр (\ref{eq9}) ЁрёїюфшЄё ,
ўЄю ш ЄЁхсютрыюё№ єёЄрэютшЄ№.
\end{proof}

\medskip
╤ыхфє■∙хх єЄтхЁцфхэшх ёяЁртхфыштю Єюы№ъю яЁш $\alpha\ne n.$

\medskip
\begin{theorem}\label{th4}
{\, ╧єёЄ№ $n\geqslant 3,$ $D$ -- юуЁрэшўхээр  юсырёЄ№ т ${\Bbb
R}^n,$  $\alpha\in (n-1, n),$ $x_0\in D,$ Єюуфр ърцфюх юЄъЁ√Єюх,
фшёъЁхЄэюх ш чрьъэєЄюх юуЁрэшўхээюх юЄюсЁрцхэшх
$f:D\setminus\{x_0\}\rightarrow {\Bbb R}^n$ ъырёёр $W_{loc}^{1,
\varphi}(D\setminus\{x_0\})$ ё ъюэхўэ√ь шёърцхэшхь Єръюх, ўЄю $C(f,
x_0)\cap C(f,
\partial D)=\varnothing,$ яЁюфюыцрхЄё  т Єюўъє $x_0$
эхяЁхЁ√тэ√ь юсЁрчюь фю юЄюсЁрцхэш  $f:D\rightarrow {\Bbb R}^n,$ хёыш
т√яюыэхэю єёыютшх (\ref{eqOS3.0a}) ш, ъЁюьх Єюую, эрщф╕Єё  ЇєэъЎш 
$Q\in L_{loc}^{1}(D),$ Єрър  ўЄю $K_{I,\alpha}(x, f)\leqslant Q(x)$
яЁш яюўЄш тёхї $x\in D$ ш $Q\in L_{loc}^s({\Bbb R}^n)$ яЁш эхъюЄюЁюь
$s\geqslant\frac{n}{n-\alpha}.$}
\end{theorem}

\medskip
\begin{proof}
╠юцэю ёўшЄрЄ№, ўЄю $x_0=0.$ ╟рЇшъёшЁєхь яЁюшчтюы№э√ь юсЁрчюь
$0<\varepsilon_0<\infty$ ш яюыюцшь $G:=B(0, \varepsilon_0),$
$\psi(t):=1/t.$ ╟рьхЄшь, ўЄю єърчрээр  ЇєэъЎш  $\psi$ єфютыхЄтюЁ хЄ
эхЁртхэёЄтрь $0< I(\varepsilon,
\varepsilon_0):=\int\limits_{\varepsilon}^{\varepsilon_0}\psi(t)dt <
\infty.$ ╧юърцхь Єръцх, ўЄю т ¤Єюь ёыєўрх т√яюыэхэю ёююЄэю°хэшх
\begin{equation} \label{eq4!}
\int\limits_{A(\varepsilon, \varepsilon_0,
0)}Q(x)\cdot\psi^{\alpha}(|x|) \
dm(x)\,=\,o\left(I^{\alpha}(\varepsilon, \varepsilon_0)\right)\,.
\end{equation}
╧Ёшьхэ   эхЁртхэёЄтю ├╕ы№фхЁр, сєфхь шьхЄ№
$$\int\limits_{\varepsilon<|x-b|<\varepsilon_0}\frac{Q(x)}{|x-b|^{\alpha}}
\ dm(x)\leqslant$$
\begin{equation}\label{eq13}
\leqslant
\left(\int\limits_{\varepsilon<|x-b|<\varepsilon_0}\frac{1}{|x-b|^{\alpha
q}} \ dm(x) \right)^{\frac{1}{q}}\,\left(\int\limits_{G}
Q^{q^{\prime}}(x)\ dm(x)\right)^{\frac{1}{q^{\prime}}}\,,
\end{equation}
уфх  $1/q+1/q^{\prime}=1$. ╟рьхЄшь, ўЄю яхЁт√щ шэЄхуЁры т яЁртющ
ўрёЄш эхЁртхэёЄтр (\ref{eq13}) ьюцхЄ с√Є№ яюфёўшЄрэ эхяюёЁхфёЄтхээю.
─хщёЄтшЄхы№эю, яєёЄ№ фы  эрўрыр $q^{\prime}=\frac{n}{n-\alpha}$ (ш,
ёыхфютрЄхы№эю, $q=\frac{n}{\alpha}.$) ┬тшфє ЄхюЁхь√ ╘єсшэш сєфхь
шьхЄ№:
$$
\int\limits_{\varepsilon<|x-b|<\varepsilon_0}\frac{1}{|x-b|^{\alpha
q}} \ dm(x)=\omega_{n-1}\int\limits_{\varepsilon}^{\varepsilon_0}
\frac{dt}{t}=\omega_{n-1}\log\frac{\varepsilon_0}{\varepsilon}\,.
$$
┬ юсючэрўхэш ї ыхьь√ \ref{lemma1} ь√ сєфхь шьхЄ№, ўЄю яЁш
$\varepsilon\rightarrow 0$
$$
\frac{1}{I^{\alpha}(\varepsilon,
\varepsilon_0)}\int\limits_{\varepsilon<|x-b|<\varepsilon_0}\frac{Q(x)}{|x-b|^{\alpha}}
\ dm(x)\leqslant \omega^{\frac{\alpha}{n}}_{n-1}\Vert
Q\Vert_{L^{\frac{n}{n-\alpha}}(G)}\left(\log\frac{\varepsilon_0}{\varepsilon}\right)
^{-\alpha+\frac{\alpha}{n}}\,\rightarrow 0\,,
$$
ўЄю тыхў╕Є т√яюыэхэшх ёююЄэю°хэш  (\ref{eq4!}).

\medskip
╧єёЄ№ ЄхяхЁ№ $q^{\prime}>\frac{n}{n-\alpha}$ (Є.х.,
$q=\frac{q^{\prime}}{q^{\prime}-1}$). ┬ ¤Єюь ёыєўрх
$$
\int\limits_{\varepsilon<|x-b|<\varepsilon_0}\frac{1}{|x-b|^{\alpha
q}} \ dm(x) = \omega_{n-1}\int\limits_{\varepsilon}^{\varepsilon_0}
t^{n-\frac{\alpha
q^{\prime}}{q^{\prime}-1}-1}dt\leqslant$$$$\leqslant
\omega_{n-1}\int\limits_{0}^{\varepsilon_0} t^{n-\frac{\alpha
q^{\prime}}{q^{\prime}-1}-1}dt =\frac{\omega_{n-1}}{n-\frac{\alpha
q^{\prime}}{q^{\prime}-1}}\varepsilon^{n-\frac{\alpha
q^{\prime}}{q^{\prime}-1}}_0,
$$
ш, чэрўшЄ,
$$
\frac{1}{I^{\alpha}(\varepsilon, \varepsilon_0)}
\int\limits_{\varepsilon<|x-b|<\varepsilon_0}\frac{Q(x)}{|x-b|^{\alpha}}
\ dm(x)\leqslant \Vert
Q\Vert_{L^{q^{\prime}}(G)}\left(\frac{\omega_{n-1}}{n-\frac{\alpha
q^{\prime}}{q^{\prime}-1}} \varepsilon^{n-\frac{\alpha
q^{\prime}}{q^{\prime}-1}}_0\right)^{\frac{1}{q}}\left(\log\frac{\varepsilon_0}{\varepsilon}\right)
^{-\alpha}\,,
$$
ўЄю Єръцх тыхў╕Є т√яюыэхэшх ёююЄэю°хэш  (\ref{eq4!}).
╧юыюцшь ЄхяхЁ№ т (\ref{eq10A}) $\eta(t)=\psi(t)/I(r_1, r_2).$ ╥юуфр
шч (\ref{eq10A}) эхьхфыхээю ёыхфєхЄ, ўЄю $\int\limits_{r_1}^{r_2}\
\frac{dr}{r^{\frac{n-1}{\alpha-1}}q_{x_0}^{\frac{1}{\alpha-1}}(r)}\rightarrow
\infty$ яЁш $r_1\rightarrow 0.$ ┬ Єръюь ёыєўрх, чръы■ўхэшх фрээющ
ЄхюЁхь√ т√ЄхърхЄ шч ЄхюЁхь√ \ref{th1}.
\end{proof}

\medskip ╥хяхЁ№ ЁрёёьюЄЁшь эхъюЄюЁ√х яЁшьхЁ√. ╟рьхЄшь, яЁхцфх тёхую, ўЄю яЁюшчтюфэр 
$\frac{\partial f}{\partial e}(x_0)=\lim\limits_{t\rightarrow
+0}\frac{f(x_0+te)-f(x_0)}{t}$ юЄюсЁрцхэш  $f$ яю эряЁртыхэш■ $e\in
{\Bbb S}^{n-1}$ т Єюўъх хую фшЇЇхЁхэЎшЁєхьюёЄш $x_0$ ьюцхЄ с√Є№
т√ўшёыхэр яю яЁртшыє: $\frac{\partial f}{\partial
e}(x_0)=f^{\,\prime}(x_0)e.$ ╥ръшь юсЁрчюь, яєЄ╕ь яЁ ьюую т√ўшёыхэш 
ьюцэю єсхфшЄ№ё  т ёяЁртхфыштюёЄш ёыхфє■∙хую єЄтхЁцфхэш .

\medskip
\begin{proposition}\label{pr1C}
{\, ╧єёЄ№ юЄюсЁрцхэшх $f: B(0, p)\rightarrow {\Bbb R}^n$ шьххЄ тшф
$$
f(x)=\frac{x}{|x|}\rho(|x|)\,,
$$
уфх ЇєэъЎш  $\rho(t):(0, p)\rightarrow {\Bbb R}$ эхяЁхЁ√тэр ш
фшЇЇхЁхэЎшЁєхьр яюўЄш тё■фє. ╥юуфр $f$ Єръцх фшЇЇхЁхэЎшЁєхью яюўЄш
тё■фє, яЁш ¤Єюь, т ърцфющ Єюўъх $x_0$ фшЇЇхЁхэЎшЁєхьюёЄш юЄюсЁрцхэш 
$f$ т ърўхёЄтх уыртэ√ї тхъЄюЁют $e_{i_1},\ldots, e_{i_n}$ ш
$\widetilde{e_{i_1}},\ldots, \widetilde{e_{i_n}}$ ьюцэю тч Є№
$(n-1)$ ышэхщэю эхчртшёшь√ї ърёрЄхы№э√ї тхъЄюЁют ъ ёЇхЁх $S(0, r)$ т
Єюўъх $x_0,$ уфх $|x_0|=r,$ ш юфшэ юЁЄюуюэры№э√щ ъ эшь тхъЄюЁ т
єърчрээющ Єюўъх.

╤ююЄтхЄёЄтє■∙шх уыртэ√х ЁрёЄ цхэш  (эрч√трхь√х, ёююЄтхЄёЄтхээю, {\it
ърёрЄхы№э√ьш ЁрёЄ цхэш ьш} ш {\it Ёрфшры№э√ь ЁрёЄ цхэшхь}) Ёртэ√

$\lambda_{\tau}(x_0):=\lambda_{i_1}(x_0)=\ldots=\lambda_{i_{n-1}}(x_0)=\frac{\rho(r)}{r}$
ш $\lambda_{r}(x_0):=\lambda_{i_n}=\rho^{\,\prime}(r),$
ёююЄтхЄёЄтхээю.}
\end{proposition}

\medskip
╬ЄьхЄшь, ўЄю фы  уыртэ√ї ЁрёЄ цхэшщ $\lambda_{i_k},$ $k\in 1,2,
\ldots, n,$ ь√ эрьхЁхээю шёяюы№чютрыш фтющэє■ шэфхъёрЎш■, яюёъюы№ъє,
ъръ ь√ єёыютшышё№ т√°х, ъюэхўэє■ яюёыхфютрЄхы№эюёЄ№ $\lambda_i,$
$i\in 1,2,\ldots, n$ ь√ яЁхфяюырурхь тючЁрёЄр■∙хщ яю $i:$
$\lambda_1\le\lambda_2\le\ldots\le\lambda_n.$ ┼ёЄхёЄтхээю, ўЄю т
ЇшъёшЁютрээющ Єюўъх $x_0$ Ёрфшры№э√х ЁрёЄ цхэш 
$\lambda_{i_1}(x_0)=\ldots=\lambda_{i_{n-1}}(x_0)=\frac{\rho(r)}{r}$
ьюуєЄ с√Є№ эх сюы№°х ърёрЄхы№эюую ЁрёЄ цхэш 
$\lambda_{i_n}=\rho^{\,\prime}(r),$ ш эрюсюЁюЄ.

\medskip
╤ыхфє■∙хх єЄтхЁцфхэшх яюърч√трхЄ, ўЄю т єёыютш ї ЄхюЁхь \ref{th1} ш
\ref{th3} ЄЁхсютрэш  эр ЇєэъЎш■ $Q$ эхы№ч , тююс∙х уютюЁ , чрьхэшЄ№
єёыютшхь $Q\in L^p$ эш фы  ъръюую (ёъюы№ єуюфэю сюы№°юую) $p>0$ ш
фы  ы■сющ эхєс√тр■∙хщ ЇєэъЎшш $\varphi(t).$ ─ы  яЁюёЄюЄ√ ЁрёёьюЄЁшь
ёыєўрщ, ъюуфр $D={\Bbb B}^n,$ $n\geqslant 3.$

\medskip
\begin{theorem}\label{th3.10.1}{\,
╧єёЄ№ $\varphi:[0,\infty)\rightarrow[0,\infty)$ -- яЁюшчтюы№эр 
эхєс√тр■∙р  ЇєэъЎш . ─ы  ърцфюую $p\geqslant 1$ ш
$n-1<\beta\leqslant n$ ёє∙хёЄтє■Є ЇєэъЎш  $Q:{\Bbb B}^n\rightarrow
[1, \infty],$ $Q(x)\in L^p({\Bbb B}^n)$ ш ЁртэюьхЁэю юуЁрэшўхээ√щ
уюьхюьюЁЇшчь $g:{\Bbb B}^n\setminus\{0\}\rightarrow {\Bbb R}^n,$
$g\in W_{loc}^{1, \varphi}({\Bbb B}^n\setminus\{0\}),$ шьх■∙шщ
ъюэхўэюх шёърцхэшх, Єръющ ўЄю $K_{I, \beta}(x, f)\leqslant Q(x),$
яЁш ¤Єюь, $g$ эх яЁюфюыцрхЄё  яю эхяЁхЁ√тэюёЄш т Єюўъє $x_0=0.$}
\end{theorem}

\medskip
\begin{proof} ╨рёёьюЄЁшь ёыхфє■∙шщ яЁшьхЁ.
╟рЇшъёшЁєхь ўшёыр $p\geqslant 1$ ш $\alpha\in \left(0,
n/p(n-1)\right).$ ╠юцэю ёўшЄрЄ№, ўЄю $\alpha<1$ т ёшыє
яЁюшчтюы№эюёЄш т√сюЁр $p.$ ╟рфрфшь уюьхюьюЁЇшчь $g:{\Bbb
B}^n\setminus\{0\}\rightarrow {\Bbb R}^n$ ёыхфє■∙шь юсЁрчюь:
$g(x)=\frac{1+|x|^{\alpha}}{|x|}\cdot x.$
╟рьхЄшь, ўЄю юЄюсЁрцхэшх $g$ яхЁхтюфшЄ °рЁ $D={\Bbb B}^n$ т ъюы№Ўю
$D^{\,\prime}=B(0,2)\setminus {\Bbb B}^n,$ яЁш ¤Єюь, $C(g, 0)={\Bbb
S}^{n-1}$ (юЄё■фр т√ЄхърхЄ, ўЄю $g$ эх шьххЄ яЁхфхыр т эєых).
╟рьхЄшь, ўЄю $g\in C^{1}({\Bbb B}^n\setminus \{0\}),$ т ўрёЄэюёЄш,
$g\in W_{loc}^{1,1}.$

\medskip
─рыхх, т ърцфющ Єюўъх $x\in {\Bbb B}^n\setminus \{0\}$ юЄюсЁрцхэш 
$g:{\Bbb B}^n\setminus \{0\}\rightarrow {\Bbb R}^n$ т√ўшёышь
тэєЄЁхээ■■ фшырЄрЎш■ юЄюсЁрцхэш  $g$ т Єюўъх $x$ яюЁ фър $\beta,$
тюёяюы№чютрт°шё№ яЁртшыюь (\ref{eq41}). ╧юёъюы№ъє $g$ шьххЄ тшф
$g(x)=\frac{x}{|x|}\rho(|x|),$ яЁ ь√ь яюфёў╕Єюь ёююЄтхЄёЄтє■∙шї
яЁюшчтюфэ√ї яю эряЁртыхэш■ ьюцэю єсхфшЄ№ё , ўЄю т ърўхёЄтх уыртэ√ї
тхъЄюЁют $e_{i_1},\ldots, e_{i_n}$ ш $\widetilde{e_{i_1}},\ldots,
\widetilde{e_{i_n}}$ ьюцэю тч Є№ $(n-1)$ ышэхщэю эхчртшёшь√ї
ърёрЄхы№э√ї тхъЄюЁют ъ ёЇхЁх $S(0, r)$ т Єюўъх $x_0,$ уфх $|x_0|=r,$
ш юфшэ юЁЄюуюэры№э√щ ъ эшь тхъЄюЁ т єърчрээющ Єюўъх. ╤ююЄтхЄёЄтє■∙шх
уыртэ√х ЁрёЄ цхэш  (эрч√трхь√х, ёююЄтхЄёЄтхээю, {\it ърёрЄхы№э√ьш
ЁрёЄ цхэш ьш} ш {\it Ёрфшры№э√ь ЁрёЄ цхэшхь}) Ёртэ√
$\lambda_{\tau}(x_0):=\lambda_{i_1}(x_0)=\ldots=\lambda_{i_{n-1}}(x_0)=\frac{\rho(r)}{r}$
ш $\lambda_{r}(x_0):=\lambda_{i_n}=\rho^{\,\prime}(r),$
ёююЄтхЄёЄтхээю.

╤юуырёэю ёърчрээюьє,
$$\lambda_{\tau}(x)=\frac{|x|^{\alpha}+1}{|x|},\lambda_{r}(x)=\alpha|x|^{\alpha-1},
l(g^{\,\prime}(x))=\alpha|x|^{\alpha-1}\,, \Vert
g^{\,\prime}(x)\Vert=\frac{|x|^{\alpha}+1}{|x|}\,,$$
$$|J(x, g)|=
\left(\frac{|x|^{\alpha}+1}{|x|}\right)^{n-1}\cdot
\alpha|x|^{\alpha-1}$$ ш
$K_{I, \beta}(x, g)=c(\alpha)\cdot\frac{(1+|x|^{\,\alpha})^{n-1}}{
|x|^{\,(\alpha-1)(\beta-1)+n-1}}.$
╟рьхЄшь, ўЄю хёыш $G$ -- яЁюшчтюы№эр  ъюьяръЄэр  юсырёЄ№ т ${\Bbb
B}^n\setminus\{0\},$ Єю $\Vert g^{\,\prime}(x)\Vert\leqslant
c(G)<\infty,$ ъЁюьх Єюую, эхЄЁєфэю тшфхЄ№, ўЄю $|\nabla
g(x)|\leqslant n^{1/2}\cdot\Vert g^{\,\prime}(x)\Vert$ яЁш яюўЄш
тёхї $x\in {\Bbb B}^n\setminus\{0\}.$ ╥юуфр ттшфє эхєс√трэш  ЇєэъЎшш
$\varphi$ т√яюыэхэю:
$\int\limits_{G} \varphi(|\nabla
g(x)|)dm(x)\leqslant\varphi(n^{1/2}c(G))\cdot m(G)<\infty,$ Є.х.,
$g\in W^{1, \varphi}(G).$
╟рьхЄшь, ўЄю юЄюсЁрцхэшх $g$ шьххЄ ъюэхўэюх шёърцхэшх, яюёъюы№ъє хую
 ъюсшрэ яюўЄш тё■фє эх Ёртхэ эєы■; ъЁюьх Єюую, $K_{I, \beta}(x,
g)=c(\alpha)\cdot\frac{(1+|x|^{\,\alpha})^{n-1}}{
|x|^{\,(\alpha-1)(\beta-1)+n-1}}.$ ╧юырурхь:
$Q=\frac{(1+|x|^{\,\alpha})^{n-1}}{
|x|^{\,(\alpha-1)(\beta-1)+n-1}},$ Єюуфр
$Q(x)\leqslant \frac{C}{|x|^{\alpha(n-1)}}.$ ╥ръшь юсЁрчюь,
яюыєўрхь:
$$\int\limits_{{\Bbb B}^n}\left(Q(x)\right)^p dm(x)\leqslant C^p
\int\limits_{{\Bbb B}^n}\frac{dm(x)}{|x|^{p\alpha(n-1)}}=$$
\begin{equation}\label{eq2.3A}=C^p\int\limits_0^1\int\limits_{S(0,
r)}\frac{d{\mathcal{A}}}{|x|^{p\alpha(n-1)}}\,dr=\omega_{n-1}C^p
\int\limits_0^1\frac{dr}{r^{(n-1)(p\alpha-1)}}\,.\end{equation}
╒юЁю°ю шчтхёЄэю, ўЄю шэЄхуЁры
$I:=\int\limits_0^1\frac{dr}{r^{\gamma}}$ ёїюфшЄё  яЁш $\gamma<1.$
╥ръшь юсЁрчюь, шэЄхуЁры т яЁртющ ўрёЄш ёююЄэю°хэш  (\ref{eq2.3A})
ёїюфшЄё , яюёъюы№ъє яюърчрЄхы№ ёЄхяхэш $\gamma:=(n-1)(p\alpha-1)$
єфютыхЄтюЁ хЄ єёыютш■ $\gamma<1$ яЁш $\alpha\in (0, n/p(n-1)).$
╬Єё■фр т√ЄхърхЄ, ўЄю $Q(x)\in L^p({\Bbb B}^n).$
\end{proof}

\medskip
╤ыхфє■∙хх єЄтхЁцфхэшх ёюфхЁцшЄ т ёхсх чръы■ўхэшх ю Єюь, ўЄю єёыютшх
(\ref{eq9})  ты хЄё  эх Єюы№ъю фюёЄрЄюўэ√ь, эю т эхъюЄюЁюь ёь√ёых ш
эхюсїюфшь√ь єёыютшхь тючьюцэюёЄш эхяЁхЁ√тэюую яЁюфюыцхэш 
юЄюсЁрцхэш  т шчюышЁютрээє■ уЁрэшўэє■ Єюўъє.

\medskip
\begin{theorem}\label{th5} { ╧єёЄ№ $\varphi:[0,\infty)\rightarrow[0,\infty)$ -- яЁюшчтюы№эр 
эхєс√тр■∙р  ЇєэъЎш , ш $n-1<\alpha\leqslant n$ ш
$0<\varepsilon_0<1.$ ─ы  ърцфющ шчьхЁшьющ яю ╦хсхує ЇєэъЎшш $Q:{\Bbb
B}^n\rightarrow [1, \infty],$ $Q\in L_{loc}({\Bbb B}^n),$ Єръющ, ўЄю
$\int\limits_{0}^{\varepsilon_0}\frac{dt}{t^{\frac{n-1}{\alpha-1}}
q_{0}^{\,\frac{1}{\alpha-1}}(t)}<\infty,$ эрщф╕Єё  юуЁрэшўхээюх
юЄюсЁрцхэшх $f\in W_{loc}^{1, \varphi}({\Bbb B}^n\setminus\{0\})$ ё
ъюэхўэ√ь шёърцхэшхь, ъюЄюЁюх эх ьюцхЄ с√Є№ яЁюфюыцхэю т Єюўъє
$x_0=0$ эхяЁхЁ√тэ√ь юсЁрчюь, яЁш ¤Єюь, $K_{I, \alpha}(x, f)\leqslant
\widetilde{Q}(x)$ я.т., уфх $\widetilde{Q}(x)$ -- эхъюЄюЁр 
шчьхЁшьр  яю ╦хсхує ЇєэъЎш , Єрър  ўЄю
$$\widetilde{q}_0(r):=\frac{1}{\omega_{n-1}r^{n-1}}\int\limits_{S(0,
r)}\widetilde{Q}(x)d{\mathcal H}^{n-1}=q_0(r)$$ фы  яюўЄш тёхї $r\in
(0, 1).$}
\end{theorem}

\begin{proof} ╤эрўрыр ЁрёёьюЄЁшь ёыєўрщ $\alpha=n.$
╬яЁхфхышь юЄюсЁрцхэшх $f:{\Bbb B}^n\setminus\{0\}\rightarrow {\Bbb
R}^n$ ёыхфє■∙шь юсЁрчюь: $f(x)=\frac{x}{|x|}\rho(|x|),$ уфх
$\rho(r)=\exp\left\{-\int\limits_{r}^1\frac{dt}{tq_{0}^{1/(n-1)}(t)}\right\}.$
╟рьхЄшь, ўЄю $f\in ACL$ ш юЄюсЁрцхэшх $f$ фшЇЇхЁхэЎшЁєхью яюўЄш
тё■фє т ${\Bbb B}^n\setminus\{0\}.$ ┬тшфє Єхїэшъш, шчыюцхээющ яхЁхф
ЇюЁьєышЁютъющ ыхьь√ \ref{thOS4.1},
$$\Vert
f^{\,\prime}(x)\Vert=\frac{\exp\left\{-\int\limits_{|x|}^1\frac{dt}{tq_{0}^{1/(n-1)}(t)}\right\}}{|x|}\,,
l(f^{\,\prime}(x))=\frac{\exp\left\{-\int\limits_{|x|}^1\frac{dt}{tq_{0}^{1/(n-1)}(t)}\right\}}
{|x|q_{0}^{1/(n-1)}(|x|)}$$ ш $|J(x, f)|=\frac{\exp\left\{-n\int
\limits_{|x|}^1\frac{dt}{tq_{0}^{1/(n-1)}(t)}\right\}}{|x|^nq_{0}^{1/(n-1)}(|x|)}.$
╟рьхЄшь, ўЄю $J(x, f)\ne 0$ яЁш яюўЄш тёхї $x,$ чэрўшЄ, $f$ --
юЄюсЁрцхэшх ё ъюэхўэ√ь шёърцхэшхь. ╩Ёюьх Єюую, юЄьхЄшь, ўЄю
$\varphi(|\nabla f(x)|)\in L_{loc}^1({\Bbb B}^n\setminus\{0\}),$
яюёъюы№ъє $\Vert f^{\,\prime}(x)\Vert$ ыюъры№эю юуЁрэшўхэр т ${\Bbb
B}^n\setminus\{0\},$ р $\varphi$ -- эхєс√тр■∙р  ЇєэъЎш . ╧єЄ╕ь
эхяюёЁхфёЄтхээ√ї т√ўшёыхэшщ єсхцфрхьё , ўЄю $K_I(x, f)=q_{0}(|x|).$
╧юырурхь $\widetilde{Q}(x):=q_{0}(|x|),$ Єюуфр сєфхь шьхЄ№, ўЄю
$\widetilde{q}_0(r)=q_0(r)$ фы  яюўЄш тёхї $r\in (0, 1).$ ═ръюэхЎ,
чрьхЄшь, ўЄю юЄюсЁрцхэшх $f$ эх яЁюфюыцрхЄё  яю эхяЁхЁ√тэюёЄш т
Єюўъє $x_0=0$ ттшфє єёыютш 
$\int\limits_{0}^{\varepsilon_0}\frac{dt}{tq_{0}^{\,\frac{1}{n-1}}(t)}<\infty.$

\medskip ╥хяхЁ№ ЁрёёьюЄЁшь ёыєўрщ $\alpha\in (n-1, n).$
╬яЁхфхышь юЄюсЁрцхэшх $f:{\Bbb B}^n\setminus\{0\}\rightarrow {\Bbb
R}^n$ ёыхфє■∙шь юсЁрчюь: $f(x)=\frac{x}{|x|}\rho(|x|),$ уфх
$$\rho(x)=\left(1+\frac{n-\alpha}{\alpha-1}\int\limits_{|x|}^1
\frac{dt}{t^{\frac{n-1}{\alpha-1}}q_0^{\frac{1}{\alpha-1}}(t)}\right)^{\frac{\alpha-1}{\alpha-n}}\,.$$
╟рьхЄшь, ўЄю $f\in ACL$ ш юЄюсЁрцхэшх $f$ фшЇЇхЁхэЎшЁєхью яюўЄш
тё■фє т ${\Bbb B}^n\setminus\{0\}.$ ┬тшфє Єхїэшъш, шчыюцхээющ яхЁхф
ЇюЁьєышЁютъющ ыхьь√ \ref{thOS4.1},
$$\Vert f^{\,\prime}(x)\Vert =\left(1+\frac{n-\alpha}{\alpha-1}\int\limits_{|x|}^1
\frac{dt}{t^{\frac{n-1}{\alpha-1}}q_0^{\frac{1}{\alpha-1}}(t)}\right)^{\frac{\alpha-1}{\alpha-n}}\cdot\frac{1}
{|x|}\,,$$
$$l(f^{\,\prime}(x))=\left(1+\frac{n-\alpha}{\alpha-1}\int\limits_{|x|}^1
\frac{dt}{t^{\frac{n-1}{\alpha-1}}q_0^{\frac{1}{\alpha-1}}(t)}\right)^{\frac{n-1}{\alpha-n}}\cdot\frac{1}
{|x|^{\frac{n-1}{\alpha-1}}q_0^{\frac{1}{\alpha-1}}(|x|)}$$
ш
$$J(x, f)=\left(1+\frac{n-\alpha}{\alpha-1}\int\limits_{|x|}^1
\frac{dt}{t^{\frac{n-1}{\alpha-1}}q_0^{\frac{1}{\alpha-1}}(t)}\right)^{\frac{(n-1)\alpha}{\alpha-n}}\cdot\frac{1}
{|x|^{n-1+\frac{n-1}{\alpha-1}}q_0^{\frac{1}{\alpha-1}}(|x|)}\,.$$
╟рьхЄшь, ўЄю $J(x, f)\ne 0$ яЁш яюўЄш тёхї $x,$ чэрўшЄ, $f$ --
юЄюсЁрцхэшх ё ъюэхўэ√ь шёърцхэшхь. ╩Ёюьх Єюую, юЄьхЄшь, ўЄю
$\varphi(|\nabla f(x)|)\in L_{loc}^1({\Bbb B}^n\setminus\{0\}),$
яюёъюы№ъє $\Vert f^{\,\prime}(x)\Vert$ ыюъры№эю юуЁрэшўхэр т ${\Bbb
B}^n\setminus\{0\},$ р $\varphi$ -- эхєс√тр■∙р  ЇєэъЎш . ╧єЄ╕ь
эхяюёЁхфёЄтхээ√ї т√ўшёыхэшщ єсхцфрхьё , ўЄю $K_I(x, f)=q_{0}(|x|).$
╧юырурхь $\widetilde{Q}(x):=q_{0}(|x|),$ Єюуфр сєфхь шьхЄ№, ўЄю
$\widetilde{q}_0(r)=q_0(r)$ фы  яюўЄш тёхї $r\in (0, 1).$ ═ръюэхЎ,
чрьхЄшь, ўЄю юЄюсЁрцхэшх $f$ эх яЁюфюыцрхЄё  яю эхяЁхЁ√тэюёЄш т
Єюўъє $x_0=0$ ттшфє єёыютш 
$\int\limits_{0}^{\varepsilon_0}\frac{dt}{t^{\frac{n-1}{\alpha-1}}q_{0}^{\,\frac{1}{\alpha-1}}(t)}<\infty.$
\end{proof}

╩╬═╥└╩╥═└▀ ╚═╘╬╨╠└╓╚▀

\medskip
\noindent{{\bf ╧хЄЁют ┼тухэшщ └ыхъёрэфЁютшў}\\
╚эёЄшЄєЄ яЁшъырфэющ ьрЄхьрЄшъш ш ьхїрэшъш ═└═ ╙ъЁршэ√, ъюьэ. 417,
єы. ухэхЁрыр ┴рЄ■ър, 19, у. ╤ырт эёъ, ╙ъЁршэр, 84 116,

\noindent e-mail: eugeniy.petrov@gmail.com

\medskip
\noindent{\bf ╤рышьют ╨єёырэ ╨рфшъютшў} \\
╚эёЄшЄєЄ ьрЄхьрЄшъш ═└═ ╙ъЁршэ√, єы. ╥хЁх∙хэъютёър  3, у. ╩шхт-4,
╙ъЁршэр, 01601, Єхы. +38 0956308592, e-mail: ruslan623@yandex.ru

\medskip
\noindent{\bf ╤хтюёЄ№ эют ┼тухэшщ └ыхъёрэфЁютшў }
\\
╞шЄюьшЁёъшщ уюёєфрЁёЄтхээ√щ єэштхЁёшЄхЄ шьхэш ╚трэр ╘Ёрэъю, єы.
┴юы№°р  ┴хЁфшўхтёър  40, у. ╞шЄюьшЁ, ╙ъЁршэр, 10008, Єхы. +38
0669595034, e-mail: esevostyanov2009@mail.ru

\end{document}